\documentclass[11pt, a4]{amsart}
\usepackage{amsmath, amsthm, amscd, amsfonts, amssymb, graphicx, color}
\usepackage[bookmarksnumbered, colorlinks, plainpages]{hyperref}

\usepackage[english]{babel}
\usepackage{amsfonts}
\usepackage{amsmath}
\usepackage[latin1]{inputenc}
\usepackage{latexsym}
\usepackage{amssymb}

\newtheorem{theorem}{Theorem}[section]

\newtheorem{conjecture}[theorem]{Conjecture}
\newtheorem{corollary}[theorem]{Corollary}

\newtheorem{lemma}[theorem]{Lemma}

\newtheorem{problem}[theorem]{Problem}
\newtheorem{proposition}[theorem]{Proposition}
\newtheorem{remark}[theorem]{Remark}

\usepackage[active]{srcltx} 
\usepackage{enumerate}
\usepackage{ragged2e}

\title{A linear preserver problem on maps which are triple derivable at orthogonal pairs}

\author[A. Ben Ali Essaleh]{Ahlem Ben Ali Essaleh}
\email{ahlem.benalisaleh@gmail.com}
\address{(Current address) Department of Mathematics, Faculty of Science, University of Hafr Al Batin, P.O.Box 1803, Hafr Al Batin 31991, Saudi Arabia\\
(Permanent address) D{\'e}partement de Math{\'e}matiques, Institut Pr{\'e}paratoire aux Etudes d'Ing{\'e}nieurs de Gafsa, Universit{\'e} de Gafsa, 2112 Gafsa, Tunisia.}

\author[A.M. Peralta]{Antonio M. Peralta}
\email{aperalta@ugr.es}
\address{Departamento de An{\'a}lisis Matem{\'a}tico, Facultad de
Ciencias, Universidad de Granada, 18071 Granada, Spain.}

\keywords{C$^*$-algebra; JB$^*$-algebra; JB$^*$-triple; derivations; triple derivations; generalized derivations; maps triple derivable at zero; maps triple derivable at orthogonal elements;}

\subjclass[2010]{47B49; 46L05; 46L40; 46T20; 47B48; 17C65}

\begin{document}

\maketitle

\begin{abstract} A linear mapping $T$ on a JB$^*$-triple is called triple derivable at orthogonal pairs if for every $a,b,c\in E$ with $a\perp b$ we have $$0 = \{T(a), b,c\} + \{a,T(b),c\}+\{a,b,T(c)\}.$$ We prove that for each bounded linear mapping $T$ on a JB$^*$-algebra $A$ the following assertions are equivalent:\begin{enumerate}[$(a)$]\item $T$ is triple derivable at zero;
\item $T$ is triple derivable at orthogonal elements;
\item There exists a Jordan $^*$-derivation $D:A\to A^{**}$, a central element $\xi\in A^{**}_{sa},$ and an anti-symmetric element $\eta$ in the multiplier algebra of $A$, such that $$ T(a) = D(a) + \xi \circ a + \eta \circ a, \hbox{ for all } a\in A;$$
\item There exist a triple derivation $\delta: A\to A^{**}$ and a symmetric element $S$ in the centroid of $A^{**}$ such that $T= \delta +S$.
\end{enumerate}  The result is new even in the case of C$^*$-algebras. We next establish a new characterization of those linear maps on a JBW$^*$-triple which are triple derivations in terms of a good local behavior on Peirce 2-subspaces. We also prove that assuming some extra conditions on a JBW$^*$-triple $M$, the following statements are equivalent for each bounded linear mapping $T$ on $M$:\begin{enumerate}[$(a)$]\item $T$ is triple derivable at orthogonal pairs;
\item There exists a triple derivation $\delta: M\to M$ and an operator $S$ in the centroid of $M$ such that $T = \delta + S$.
\end{enumerate}
\end{abstract}

\section{Introduction:}

\justify

Suppose $X$ is a Banach $A$-bimodule over a complex Banach algebra $A$. A derivation from $A$ onto $X$ is a linear mapping $D: A\to X$ satisfying the following algebraic identity \begin{equation}\label{eq fundamental identity of derivations} D(a b) = D(a) b + a D(b), \ \ \forall(a,b)\in A^2.
\end{equation}

Researchers working on preservers problems are recently exploring the idea of finding conditions, weaker than those in the original definition, under which a (continuous) linear mapping is a derivation, a (Jordan) homomorphism, etc. For example,
\begin{problem}\label{problem general problem}  Suppose $T: A\to X$ is a linear map satisfying \eqref{eq fundamental identity of derivations} only on a proper subset $\mathfrak{D}\subset A^2$. Is $T$ a derivation?
\end{problem}

There is no need to comment that the role of the set $\mathfrak{D}$ is the real core of the question. A linear map $T: A\to X$ is said to be a \emph{derivation} or \emph{derivable at a point} $z\in A$ if the identity \eqref{eq fundamental identity of derivations} holds for every $(a,b)\in\mathfrak{D}_z:=\{ (a,b)\in A^2 : a b =z\}$. In order to illustrate the interest on linear maps on Banach algebras which are derivable at a concrete point, the reader can consult the references \cite{Bre07,ChebKeLee,Jing,JingLuLi2002,Lu2009,ZhuXiLi}, among other additional papers.\smallskip

H. Ghahramani, Z. Pan \cite{GhaPan2018} and B. Fadaee and H. Ghahramani \cite{FaaGhahra2019} have recently considered certain variants of Problem \ref{problem general problem} in their studies on continuous linear operators from a C$^*$-algebra $A$ into a Banach $A$-bimodule $X$ behaving like derivations or anti-derivations at elements in a certain subset of $A^2$ determined by orthogonality conditions. Let us detail the problem.

\begin{problem}\label{problem anti-derivable mapps} Let $T:A\to X$ be a continuous linear operator which is anti-derivable at zero, i.e.,  \begin{equation}\label{eq antider} \hbox{ $T(ab) = T(b) a + b T(a)$  for all $(a,b)\in \mathfrak{D}_0$.}
\end{equation}  Is $T$ an anti-derivation or a perturbation of an anti-derivation?
\end{problem}

Clearly, a mapping $\mathcal{D} : A\to X$ is called an \emph{anti-derivation} if the identity \eqref{eq antider} holds for every $(a,b)\in A^2$. If $A$ is a C$^*$-algebra, a $^*$-derivation (respectively, a $^*$-anti-derivation) from $A$ into itself, or into $A^{**}$, is a derivation (respectively, an anti-derivation) $d: A\to A$ satisfying $d(a^*) = d(a)^*$ for all $a\in A$.\smallskip

Concerning Problem \ref{problem general problem}, B. Fadaee and H. Ghahramani prove in \cite[Theorem 3.1]{FaaGhahra2019} that for a continuous linear map $T : A \to A^{**}$, where $A$ is a C$^*$-algebra, the following statement holds: $T$ is derivable at zero if and only if there is a continuous derivation $d : A \to  A^{**}$ and an element $\eta\in  \mathcal{Z}(A^{**})$ (the centre of $A^{**}$) such that $T(a) = d(a) + \eta a$ for all $a \in A$. They also obtain a similar conclusion when $T$ is r-$^*$-derivable at zero (that is, $ab^* = 0 \Rightarrow$ $a T(b)^* + T(a) b^* = 0$). Similar results were established by H. Ghahramani and Z. Pan for linear maps on a unital $^*$-algebra which is zero product determined (cf. \cite[Theorem 3.1]{GhaPan2018}).\smallskip


B. Fadaee and H. Ghahramani also study continuous linear maps from a C$^*$-algebra $A$ into its bidual which are anti-derivable at zero (see \cite[Theorem 3.3]{FaaGhahra2019}). In \cite{AbulJamPe20}, D.A. Abulhamail, F.B. Jamjoom and the second author of this note prove that for each bounded linear operator $T$ from a C$^*$-algebra $A$ into an essential Banach $A$-bimodule $X$ the following statements are equivalent:\begin{enumerate}[$(a)$]\item $T$ is anti-derivable at zero (i.e. $ab =0$ in $A$ implies $T(b) a + b T(a)=0$);
\item There exist an anti-derivation $d:A\to X^{**}$ and an element $\xi \in X^{**}$ satisfying $\xi a = a \xi,$ $\xi [a,b]=0,$
$T(a b) = b T(a) + T(b) a - b \xi a,$ and $T(a) = d(a) + \xi a,$ for all $a,b\in A$.
\end{enumerate} The conclusion can be improved in the case in which $A$ is unital, or if $X$ is a dual Banach $A$-bimodule.\smallskip

A linear mapping $\delta$ on a JB$^*$-triple $E$ with triple product $\{.,.,.\}$ is called a triple derivation if it satisfies the following triple Leibniz's rule $$\delta\{a,b,c\} = \{\delta(a), b, c\} + \{a,\delta(b), c\} + \{a,b,\delta(c)\},$$ for all $a,b,c\in E$ (see subsection \ref{subsect: background} for the concrete definitions and relations between the basic notions). The class of triple derivations have been intensively studied (see, for example, \cite{BarFri90,HoMarPeRu,HoPeRu13,PeRu14}). Surjective linear isometries between C$^*$-algebras need not be, in general, C$^*$-homomorphisms nor Jordan $^*$-homomorphisms (cf. the famous paper \cite{Kad51}). However, a remarkable result by W. Kaup proves that a linear surjection between JB$^*$-triples is an isometry if and only if it is a triple isomorphism (cf. \cite[Proposition 5.5]{Ka83}). Since early papers in the theory of JB$^*$-triple up to recently contributions, triple derivations can be applied to define one-parameter continuous semigroups of surjective linear isometries, equivalently, triple isomorphisms on JB$^*$-triples (see, for example, \cite{GarPe20}).\smallskip

A linear mapping $T$ on a JB$^*$-triple $E$ is said to be \emph{triple derivable at an element $z\in E$} if $$T(z)= T\{a,b,c\} = \{T(a), b, c\} + \{a,T(b), c\} + \{a,b,T(c)\},$$ for all $a,b,c\in E$ with $\{a,b,c\}=z$. As we shall see in subsection \ref{subsect: background}, C$^*$-algebras are one of the first natural examples of JB$^*$-triples. In this particular setting, we recently proved in \cite[Theorem 2.3 and Corollary 2.5]{EssPe2018} that every continuous linear map on a unital C$^*$-algebra $A$ which is a triple derivation at the unit of $A$ is a triple derivation and a generalized derivation, that is, $$T(a b) = T(a) b + a T(b) - a T(1) b,$$ for all $a,b\in A$. If we additionally assume that $T(1)=0$, then $T$ is a $^*$-derivation and a triple derivation \cite[Proposition 2.4]{EssPe2018}. We also know from \cite[Theorem 2.9, Corollary 2.10]{EssPe2018} that every bounded linear map on a unital C$^*$-algebra which is triple derivable at zero must be a generalized derivation, if we further assume that $T(1)=0$ then $T$ is a $^*$-derivation and a triple derivation too. Corollary 2.11 in \cite{EssPe2018} proves that every bounded linear map $T$ on a unital C$^*$-algebra which is triple derivable at zero and $T(1)^*=-T(1)$ must be a triple derivation. A more surprising conclusion is that every linear map on a von Neumann algebra which is triple derivable at zero is automatically continuous (see \cite[Corollary 2.14]{EssPe2018}).\smallskip

In this paper we shall extend the above results in two different directions. Firstly, we shall consider linear maps whose domains will be in the wider (and more natural) classes of JB$^*$-algebras and JB$^*$-triples. Secondly, we shall consider a hypothesis which is, a priori, weaker than being triple derivable at zero.  The relation ``being orthogonal'' among elements in C$^*$-algebras, JB$^*$-algebras and JB$^*$-triples has turned out to an useful to understand and classify these structures through maps preserving zero products and orthogonality (see, for example,
\cite{AlBreExVill09,Bur13,BurFerGarMarPe,BurFerGarPe09,BurSanchOr,GarPe20,
LeuWong10,LiuChouLiaoWong2018b,LiuChouLiaoWong2018,
Wong2007}, among others). As we shall detail later, elements $a,b$ in a JB$^*$-triple $E$ are orthogonal if $L(a,b)=\{a,b,\cdot\}=0$, equivalently, $\{a,b,x\} =0$ for all $x\in E$.\smallskip

We shall say that a linear mapping $T$ on a JB$^*$-triple $E$ is \emph{triple derivable at orthogonal pairs} if $$0=T\{a,b,c\} = \{T(a),b,c\} + \{a,T(b),c\} + \{a,b,T(c)\}$$ only for those $a,b,c\in E$ with $a\perp b$. Clearly, every triple derivation on $E$ is triple derivable at orthogonal pairs, and every linear map on $E$ which is triple derivable at zero is triple derivable at orthogonal pairs. In order to present new examples, let us recall that the centroid, $Z(E),$ of a JB$^*$-triple $E$ is the set of all continuous linear operators $S: E\to E$ satisfying $S\{a,b,c\} =\{S(c),b,a\}$ for all $a,b,c\in E$ (see \cite{DiTim88}). By \cite[Lemma 2.6]{DiTim88} for each $S\in Z(E)$ there exists a unique bounded linear operator $R: E\to E$ satisfying $\{a,S(b),c\} = R \{a,b,c\}$ for all $a,b,c\in E$. Let $\delta: E\to E$ be a triple derivation and $S\in Z(E)$. It is easy to see that the mapping $T= \delta + S$ satisfies that for each $a,b,c\in E$ with $a\perp b$ we have $$ 0 = \{T(a),b,c\} + \{a,T(b),c\}+\{a,b,T(c)\}.$$ That is, $T$ is triple derivable at orthogonal pairs. We are naturally led to the following conjecture:

\begin{conjecture}\label{conjecture 1}
Every continuous linear mapping $T: E\to E$ which is triple derivable at orthogonal pairs must be of the form $T = \delta +S$, where $S$ is a triple derivation on $E$ and $S$ is an element in the centroid of $E$.
\end{conjecture}

In this paper we provide a complete positive proof of Conjecture \ref{conjecture 1} in the case in which the JB$^*$-triple $E$ is a JB$^*$-algebra (see Theorem \ref{t bd linear maps triple derivable at orth pairs on a JBstar algebra}). We also study this conjecture for general JB$^*$-triples. In section 2 we shall present some new applications of the contractive projection principle to give a first characterization of those bounded linear maps on a JB$^*$-triple which are triple derivations. A more general characterization is established in Theorem \ref{t characterization of triple derivations}, where we prove that for any bounded linear map $T$ on a JB$^*$-triple $E$ in which tripotents are norm-total, the following statements are equivalent:
\begin{enumerate}[$(a)$] \item $T$ is a triple derivation;
\item $T$ is triple derivable at orthogonal elements and for each tripotent $e$ in $E$, the element $P_2(e)T(e)$ is a skew symmetric element in the JB$^*$-algebra $E_2(e)$ {\rm(}i.e. $(P_2(e)T(e))^{*_e} = -T(e)${\rm)}.
\end{enumerate}

Section \ref{sec:5} is devoted to the study of Conjecture \ref{conjecture 1}. Our main conclusion (see Theorem \ref{t triple derivable maps on JBWstar triples}) affirms that if $M$ is a JBW$^*$-triple such that linearity on it is determined by Peirce 2-subspaces, the the following statements are equivalent for any bounded linear map $T$ on $M$: \begin{enumerate}[$(a)$]\item $T$ is triple derivable at orthogonal pairs;
\item There exists a triple derivation $\delta: M\to M$ and an operator $S$ in the centroid of $M$ such that $T = \delta + S$.
\end{enumerate}

The existence of JB$^*$-triples admitting no non-trivial pairs of orthogonal elements (i.e. JB$^*$-triples or rank one), makes Conjecture \ref{conjecture 1} nonviable. It should be noted that if a JBW$^*$-triple satisfies that linearity on it is determined by Peirce 2-subspaces, then it cannot have rank one. By applying that Cartan factors have a very trivial centroid, we also show that if $M = \bigoplus_{j}^{\infty} C_j$ is an atomic JBW$^*$-triple such that each $C_j$ is a Cartan factor with rank at least $2$, then a bounded linear mapping $T: M\to M$ is triple derivable at orthogonal pairs if, and only if, there exists a triple derivation $\delta: M\to M$ and an operator $S$ in the centroid of $M$ such that $T = \delta + S$ (see Theorem \ref{t triple derivable maps on atomic JBWstar triples}).

\subsection{Basic background}\label{subsect: background}

A \emph{Jordan algebra} $A$ is a (non-necessarily associative) algebra whose product is abelian and satisfies the so-called \emph{Jordan identity} $$(a \circ b)\circ a^2 = a\circ (b \circ a^2), \ \ (a,b\in A).$$ A \emph{Jordan Banach algebra} is a Jordan algebra $A$ equipped with a complete norm, $\|.\|$, satisfying $\| a\circ b\| \leq \|a\| \ \|b\|$ ($a,b\in A$). Every real or complex associative Banach algebra is a real Jordan Banach algebra with respect to the natural Jordan product $a\circ b: = \frac12 (a b +ba)$. A \emph{JB$^*$-algebra} is a complex Jordan Banach algebra $A$ equipped
with an algebra involution $^*$ satisfying  $\|\{ a,{a},a\} \|= \|a\|^3$, $a\in A,$ where $\{a,{a},a\}  = 2 (a\circ a^*) \circ a - a^2 \circ a^*$. C$^*$-algebras are among the examples of JB$^*$-algebras when they are equipped with the natural Jordan product and the natural norm and involution. The references \cite{HOS,Cabrera-Rodriguez-vol1,Wri77} can be consulted as reference guides for the basic notions and results in the theory of JB$^*$-algebras.\smallskip

C$^*$- and JB$^*$-algebras are particular examples of Banach spaces in a wider class of complex Banach spaces known under the name of JB$^*$-triples. A \emph{JB$^*$-triple} is a complex Banach space $E$ equipped with a continuous triple product $\{ .,.,.\} : E\times E\times E \to E,$ $(a,b,c)\mapsto \{a,b,c\},$ which is bilinear and symmetric in $(a,c)$ and conjugate linear in $b$,
and satisfies the following axioms for all $a,b,x,y\in E$:
\begin{enumerate}[{\rm (a)}] \item Jordan identity: $$L(a,b) L(x,y) = L(x,y) L(a,b) + L(L(a,b)x,y)
	- L(x,L(b,a)y),$$ where $L(a,b):E\to E$ is the operator defined by $L(a,b) x = \{ a,b,x\};$
\item $L(a,a)$ is a hermitian operator with non-negative spectrum;
\item $\|\{a,a,a\}\| = \|a\|^3$.\end{enumerate}
This definition is an analytic-algebraic approach, established by W. Kaup in \cite{Ka83}, to study bounded symmetric domains in complex Banach spaces from the holomorphic point of view. Concretely, for each abstract bounded symmetric domain $\mathcal{D}$ in an arbitrary complex Banach space there exists a unique (up to linear isometries) JB$^*$-triple $E$ whose open unit ball is biholomorphically equivalent to $\mathcal{D}$.\smallskip

JB$^*$-triples enjoy many geometric properties, for example, a linear bijection $T$ on a JB$^*$-triple $E$ is an isometry if and only if it is a triple isomorphism, that is, $T\{a,b,c\} = \{T(a), T(b), T(c)\}$ for all $a,b,c\in E$ (cf. \cite[Proposition 5.5]{Ka83}).\smallskip

The triple products \begin{equation}\label{eq product operators} \{x,y,z\} =\frac12 (x y^* z +z y^*
x),\end{equation} and \begin{equation}\label{eq product jordan}\{ x,y,z\} = (x\circ y^*) \circ z + (z\circ y^*)\circ x -
(x\circ z)\circ y^*, \end{equation} are employed to induce a structure of JB$^*$-triple on C$^*$- and JB$^*$-algebras, respectively. The first one is also valid to define an structure of JB$^*$-triple on the space $B(H,K)$ of all bounded linear operators between two complex Hilbert spaces $H$ and $K$. Along the paper we shall write $B(X)$ for the Banach space of all bounded linear operators on a Banach space $X$. \smallskip

An element $e$ in a real or complex JB$^*$-triple $E$ is said to be a \emph{tripotent} if $\{e,e,e\}=e$. 
Each tripotent $e\in E$, determines a decomposition of $E,$ $$E= E_{2} (e) \oplus E_{1} (e) \oplus E_0 (e),$$ known as the \emph{Peirce decomposition}\label{eq Peirce decomposition} associated with $e$, where $E_j (e)=\{ x\in E : \{ e,e,x\} = \frac{j}{2}x \}$
for each $j=0,1,2$.\smallskip

Triple products among elements in the Peirce subspaces satisfy the following \emph{Peirce arithmetic}:\label{Peirce arithmetic} $\{ {E_{i}(e)},{E_{j} (e)},{E_{k} (e)}\}\subseteq E_{i-j+k} (e)$ if $i-j+k \in \{ 0,1,2\},$ and $\{ {E_{i}(e)},{E_{j} (e)},{E_{k} (e)}\}=\{0\}$ otherwise, and $$\{ {E_{2} (e)},{E_{0}(e)},{E}\} = \{ {E_{0} (e)},{E_{2}(e)},{E}\} =0.$$ Consequently, each Peirce subspace $E_j(e)$ is a JB$^*$-subtriple of $E$.\smallskip

The Peirce 2-subspace $E_2(e)$ enjoys has an additional structure. Namely, $E_2 (e)$ is a unital JB$^*$-algebra with unit $e$,
product $x\circ_e y := \{x,e,y\}$ and involution $x^{*_e} := \{ e,x,e\}$, respectively.\label{eq Peirce-2 is a JB-algebra}\smallskip

A tripotent $e$ in a JB$^*$-triple $E$ is called \emph{complete} if $E_0(e)=\{0\}$.\smallskip

Elements $a,b$ in a JB$^*$-triple $E$ are called \emph{orthogonal} (denoted by $a \perp b$) if $L(a,b) = 0$. The reader is referred to \cite[Lemma 1]{BurFerGarMarPe} for several reformulations of the relation ``being orthogonal'' which will be applied without any explicit mention. It should be also noted that elements $a,b$ in a C$^*$-algebra $A$ are orthogonal in the C$^*$- sense (i.e. $ab^* = b^* a =0$) if and only if they are orthogonal when $A$ is regarded as a JB$^*$-triple.\smallskip

We shall consider the natural partial order $\leq$ on the set of tripotents of a JB$^*$-triple $E$ defined by $e\leq u$ if $u-e$ is a tripotent with $u-e \perp e$.

\subsection{Jordan derivations}

Let $X$ be a Jordan-Banach module over a Jordan Banach algebra $A$. A \emph{Jordan derivation} from $A$ into $X$ is a linear map $D: A\to X$ satisfying: $$D(a\circ b ) = D(a)\circ b + a \circ D(b).$$ If $X$ is unital Jordan Banach module over a JB$^*$-algebra $A$, it is easy to check that $D(1) = 0$, for every Jordan derivation $D: A\to X$. A \emph{Jordan $^*$-derivation} on $A$ is a Jordan derivation $D$ satisfying $D(a)^* = D(a^*),$ for all $a\in A$. Following standard notation, given $x\in X$ and $a\in A$, the symbols $M_a$ and $M_x$ will denote the mappings $M_a : X\to X$, $x\mapsto M_a (x) =  a\circ x$ and $M_x: A\to X$, $a \mapsto M_x(a) = a\circ x$. By a little abuse of notation, we also denote by $M_a$ the operator on $A$ defined by $M_a (b)  = a\circ b$. Examples of Jordan derivations can be given as follows: if we fix $a\in A$ and $x\in X$, the mapping $$[M_x, M_a] = M_x M_a -M_a M_x : A \to X, \ b\mapsto [M_x, M_a] (b),$$ is a Jordan derivation. A derivation $D: A \to X$ which can be written in the form $D= \sum_{i=1}^{m} \left(M_{x_i} M_{a_i} -M_{a_i} M_{x_i}\right)$, ($x_i\in X, a_i\in A$) is called \emph{inner}. \smallskip


Returning to the setting to triple derivations, by the Jordan identity, given $a,b$ in a JB$^*$-triple $E$, the mapping $\delta (a,b) := L(a,b)-L(b,a)$ is a triple derivation on $E$ and obviously continuous. Actually, every triple derivation on a JB$^*$-triple is automatically continuous (see \cite[Corollary 2.2]{BarFri90}).\smallskip

A JBW$^*$-algebra is a JB$^*$-algebra which is also a dual Banach space (cf. \cite[Theorem 4.4.16]{HOS}). The second dual, $A^{**}$, of a JB$^*$-algebra $A$ is a JBW$^*$-algebra with a respect to a product extending the original Jordan product of $A$ (see \cite[Theorem 4.4.3]{HOS}). The Jordan product of every JBW$^*$-algebra is separately weak$^*$ continuous (cf. \cite[Theorem 4.4.16 and Corollary 4.1.6]{HOS}). In the setting of JB$^*$-triples, a JBW$^*$-triple is a JB$^*$-triple which is also a dual Banach space. The bidual of every JB$^*$-triple is a JBW$^*$-triple \cite{Di86}. The alter ego of Sakai's theorem asserts that every JBW$^*$-triple admits a unique (isometric) predual and its product is separately weak$^*$ continuous \cite{BaTi} (see also \cite[Theorems 5.7.20 and 5.7.38]{Cabrera-Rodriguez-vol2}).\smallskip

Let $E$ be a JB$^*$-triple. It is known that the JB$^*$-subtriple of $E$ generated by a single element $a\in A$ is (isometrically) triple isomorphic to a commutative C$^*$-algebra (cf. \cite[Corollary 4.8]{Ka77}, \cite[Corollary 1.15]{Ka83} and \cite{FriRu85}). In particular, we can find an element $y$ in this subtriple satisfying that $\{y,y,y\} =a$. The element $y,$ denoted by $a^{[\frac13 ]}$, is called the \emph{cubic root} of $a$. The $3^n$th roots of $a$ are inductively defined by $a^{[\frac{1}{3^n}]} = \left(a^{[\frac{1}{3^{n-1}}]}\right)^{[\frac 13]}$, $n\in \mathbb{N}$. If $a$ is an element in a JBW$^*$-triple $M$, the sequence $(a^{[\frac{1}{3^n}]})$ converges in the weak$^*$-topology of $M$ to a tripotent denoted by $r(a)$, which is called the \emph{range tripotent} of $a$. The tripotent $r(a)$ is the smallest tripotent $e\in M$ satisfying that $a$ is positive in the JBW$^*$-algebra $E^{**}_{2} (e)$ (compare \cite[Lemma 3.3]{EdRu96compact}).\smallskip

By the separate weak$^*$ continuity of the triple product of every JBW$^*$-triple, we can conclude that for each triple derivation $\delta$ on a JB$^*$-triple $E$, the bitranspose $\delta^{**} : E^{**}\to E^{**}$ is a triple derivation too.\smallskip

It is known that every Jordan $^*$-derivation on a JB$^*$-algebra is a triple derivation \cite[Lemma 2]{HoMarPeRu}. Reciprocally, for each skew symmetric element $a\in A$, the mapping $M_a (x) = a\circ x$ is a triple derivation, and if $\delta: A\to A$ is a triple derivation, the element $\delta (1)$ is skew symmetric and the mapping $\delta - M_{\delta(u)}$ is a Jordan $^*$-derivation on $A$ (cf. \cite[Lemma 1 and its proof]{HoMarPeRu}).\label{eq Jordan and triple derivations are related by HMPR}

\subsection{The centre of a JB$^*$-algebra and the centroid of a JB$^*$-triple}
Elements $a, b$ in a JB$^*$-algebra $A$ are said to \emph{operator commute} if the corresponding Jordan multiplication operators $M_a$ and $M_b$ commute in $B(A)$, i.e. if $(a\circ c)\circ b=  a\circ (c\circ b)$ for all $c\in A$. The centre of $A$, $Z(A)$, is the set of all elements of $A$ which operator commute with every other element of $A$. By the Shirshov-Cohn theorem for JB algebras (cf. \cite[Theorem 7.2.5]{HOS}) two self-adjoint elements $a$ and $b$ in $A$ generate
a JB$^*$-subalgebra which is a JC$^*$-subalgebra of some $B(H)$ (see also \cite{Wri77}), and, under this realization, $a$ and $b$
commute in the usual sense whenever they operator commute in $B(H)$ (see \cite[Proposition 1]{Top}). It is also known from the same sources that two elements $a$ and $b$ of $A_{sa}$ operator commute if and only if $a^2 \circ b =\{a,b,a\}$ (equivalently, $a^2 \circ b = 2 (a\circ b)\circ a - a^2 \circ b$).\label{eq operator commutativity} If a C$^*$-algebra $A$ is regarded as a JB$^*$-algebra with respect to its natural Jordan product, it follows from the above that the centre of $A$ in the Jordan sense is precisely the usual centre, i.e., the set of all $a\in A$ such that $ ab = ba $ for all $b\in A$.\smallskip

Several authors have treated the notions of commutativity and operator commutativity in C$^*$- and JB$^*$-algebras. Concerning orthogonality, B. Fadaee and H. Ghahramani showed in \cite[Lemma 2.2]{FaaGhahra2019} that given an element $\eta$ in the bidual of a C$^*$-algebra $A$, the condition $a \eta b = 0$ for all $a,b\in A$ with $a b=0$ implies that $\eta$ lies in the centre of $A^{**}$. An appropriate version for Banach bimodules over C$^*$-algebras is established in \cite[Lemma 5]{AbulJamPe20}. We present next a Jordan version of of this fact.

\begin{lemma}\label{l centrality in Jordan algebras} Let $A$ be a JB$^*$-algebra. Let $\xi$ be an element in $A^{**}$ such that $U_{a,b} (\xi) = 0$ for all $a,b\in A_{sa}$ with $a\perp b$. Then $\xi$ lies in the centre of $A^{**}$.
\end{lemma}

\begin{proof} Let us fix a functional $\phi\in A^*$. We shall consider the following symmetric bounded bilinear form $V : A\times A\to \mathbb{C}$, $V (a,b) = \phi U_{a,b} (\xi)$. It follows from our assumptions that $V (a,b)= 0,$ for every  $a,b\in A_{sa}$ with $a\perp b$. We have therefore shown that $V$ is a symmetric orthogonal form (cf. \cite[Corollary 3.14]{JamPeSidd2015}). We deduce from \cite[Theorem 3.6]{JamPeSidd2015} that the mapping $G_V : A\to A^*$, $G_V (a) (b) = V(a,b)$ is a purely Jordan generalized derivation. Furthermore, by Remark \ref{r module is the dual} there exists $\varphi \in A^*$ such that $$G_V (a\circ b) = G_V(a) \circ b + a\circ G_V(b) - U_{a,b} (\varphi),$$ for all $a,b\in A$. Then $$G_V (a\circ b) (c) = (G_V(a) \circ b + a\circ G_V(b) - U_{a,b} (\varphi))(c),$$ or equivalently,
$$V (a\circ b, c) = V(a, c\circ b) + V(b,a\circ c) - \varphi U_{a,b} (c),$$ or
$$\phi U_{a\circ b,c} (\xi) = \phi U_{a,c\circ b} (\xi) + \phi U_{b,a\circ c} (\xi) - \varphi U_{a,b} (c),$$ for all $a,b,c\in A$. If in the later expression we replace $c$ by the elements in a bounded increasing approximate unit  $(c_{\lambda})_{\lambda}$ in $A$, and apply that  $(c_{\lambda})_{\lambda}\to 1$ in the weak$^*$ topology of $A^{**}$, by taking weak$^*$ limits we have
\begin{equation}\label{eq key id lemma centre} \phi ( (a\circ b)\circ \xi) = \phi U_{a, b} (\xi) + \phi U_{b,a} (\xi) - \varphi (a\circ b),
\end{equation} for all $a,b\in A$. Replacing, $b$ with $c_{\lambda}$ and taking weak$^*$ limits we get
$$\phi ( a\circ \xi) = \phi ( 2 (a \circ \xi) ) - \varphi (a),$$ for all $a\in A$, witnessing that $\varphi (a) = \phi (a\circ \xi)$ for all $a\in A$. Combining this fact with \eqref{eq key id lemma centre} we prove that $\phi U_{a, b} (\xi) + \phi U_{b,a} (\xi) = 2 \phi ( (a\circ b)\circ \xi),$ for all $a,b\in A$. The arbitrariness of $\phi$ and the Hahn-Banach theorem prove that \begin{equation}\label{eq pre centre via U} 2 U_{a, b} (\xi)= U_{a, b} (\xi) +  U_{b,a} (\xi) = 2 (a\circ b)\circ \xi,
 \end{equation}for all $a,b\in A$. The weak$^*$-density of $A$ in $A^{**}$ (Goldstine's Theorem) and the separate weak$^*$-continuity of the product of $A^{**}$ (cf. \cite[Theorem 4.4.16 and Corollary 4.1.6]{HOS}) allow us to deduce that \eqref{eq pre centre via U} holds for all $a,b\in A^{**}$. The same is true for the hermitian and the skew symmetric part of $\xi$ in $A^{**}$, $\xi_1 = \frac{\xi+\xi^*}{2}$ and $\xi_2 = \frac{\xi-\xi^*}{2}$.\smallskip

An element $s$ in a unital JB$^*$-algebra is called a symmetry if $s^2 =1$. If we replace $a$ and $b$ with a symmetry $s\in A^{**}$ in the version of \eqref{eq pre centre via U} for $\xi_1,\xi_2$ in $A^{**}$, we see that $U_s (\xi_j) = (s\circ s)\circ \xi_j = \xi_j$ for $j=1,2$. Lemma 4.3.2 in \cite{HOS} asserts that $\xi_j$ lies in the centre of $A^{**}$ for all $j=1,2$, which concludes the proof.
\end{proof}

The lacking of a binary product in a JB$^*$-triple disables a natural notion of centre. The closest concept is the notion of centroid. Let $E$ be a JB$^*$-triple. 
The main motivation to introduce the centroid is the connection with the centralizer of a Banach space $X$. We recall that a \emph{multiplier} on $X$ is a bounded linear mapping $T:X\to X$ satisfying that each extreme point $\phi$ of the closed unit ball of $X^*$ is an eigenvector of the transposed mapping $T^*: X^*\to X^*$, that is, $T^*(\phi) = \lambda_{T} (\phi) \phi$, where $\lambda_{T} (\phi)$ is the corresponding eigenvalue (see \cite{Beh79}). The \emph{centralizer}, $C(X)$, of $X$ is the set of all multipliers $T$ on $X$ for which there exists another multiplier $S$ on $X$ satisfying $\lambda_{S} (\phi) =\overline{\lambda_{T} (\phi)}$. An attractive result by S. Dineen and R. Timoney proves that for each JB$^*$-triple $E$, the centroid of $E$ coincides with its centralizer, and it is precisely the set of all bounded linear operators on $E$ commuting with all Hermitian operators on $E$ (cf. \cite[Theorem 2.8 and Corollary 2.9]{DiTim88}). By \cite[Lemma 2.6]{DiTim88} for each $S\in Z(E)$ there exists a unique bounded linear operator $R: E\to E$ satisfying $\{a,S(b),c\} = R \{a,b,c\}$ for all $a,b,c\in E$. \smallskip

The \emph{centroid} of a JB$^*$-algebra $A$ is the set of all bounded linear operators $T$ on $A$ satisfying $$T(a\circ b) = T(a) \circ b, \hbox{ for all } a,b\in A.$$ The centroid of $A$ as JB$^*$-triple coincides with its centroid as JB$^*$-algebra, and if $A$ is unital a bounded linear operator $T$ lies in the centroid of $A$ if and only if $T = M_{z}$ for a unique element $z$ in the centre of $A$ (see \cite[Propositions 3.4 and 3.5]{DiTim88}).\smallskip

We shall gather next some other properties of the centroid. Let $T$ be an element in the centroid of a JB$^*$-triple $E$ and let $e$ be a tripotent in $E$. Since $T L(e,e) = L(e,e) T$, we can clearly conclude that $T(E_j (e))\subseteq E_j(e)$ for all $j=0,1,2$. In particular, $T|_{E_2(e)} : E_2(e) \to E_2(e)$ is a bounded linear operator, and given $a,b\in E_2(e)$ we can see that $$T(a\circ_e b) = T\{a,e,b\} = \{T(a), e, b\} = T(a)\circ_e b,$$ which proves that $T|_{E_2(e)}$ is an element in the centroid of the unital JB$^*$-algebra $E_2(e)$, therefore there exists an element $\xi_e$ in the centre of $E_2(e)$ such that $T(a) = \xi_e \circ_e a$ for all $a\in E_2(e)$.\label{centroid and Peirce 2-subspaces}

\section{Contractive projections and derivations}

Let us recall that linear maps on a JB$^*$-triple which are triple derivable at a point have not been described yet. So, linear maps which are triple derivable at orthogonal elements remain also unknown.\smallskip

A linear mapping $P$ on a Banach space $X$ is called a projection if $P^2 = P$. A projection $P$ on $X$ is contractive if $\|P\|=1$. The image of a contractive projection on a C$^*$-algebra need not be, in general, a C$^*$-algebra; for example, for a rank-one projection $p$ in $B(H)$, the mapping $P(x) = px$ is a contractive projection whose image is isometrically isomorphic to the Hilbert space $H$, which clearly is not a C$^*$-algebra whenever $H$ is infinite dimensional. However, for each contractive projection $P$ on a JB$^*$-triple $E$, its image, $P(E),$ is always a JB$^*$-triple with respect to the inherited norm and the triple product given by \begin{equation}\label{eq triple product contractive projection} \{x,y,z\}_{_P} :=P\{x,y,z\}, \ \ (x,y,z\in P(E))
\end{equation} (cf. \cite{stacho1982projection}, \cite[Theorem]{kaup1984contractive} and \cite{friedman1985solution}). The triple product given in \eqref{eq triple product contractive projection} will be denoted by $\{.,.,.\}_{_P}$, to avoid confusion the JB$^*$-triple $P(E)$ will be, in general, denoted by $(P(E),\{.,.,.\}_{_P})$. The image of a contractive projection need not be a JB$^*$-subtriple of $E$ \cite[Example 3]{kaup1984contractive}, however, the JB$^*$-triple $P(E)$ is isometrically isomorphic to a closed subtriple of $E^{**}$ \cite[Theorem 2]{FriRu87conditional}.\smallskip

Let $P: E\to E$ be a contractive projection on a JB$^*$-triple. The next identity was established by W. Kaup in \cite[Identity $(4)$ in page 97]{kaup1984contractive} \begin{equation}\label{eq Kaup contractive projection identity} P \{P(a), b, P (c) \} = P  \{P(a), P (b), P (c) \}, \ \ \ (a,b,c\in E).
 \end{equation} Y. Friedman and B. Russo complemented Kaup's result by showing that the identity \begin{equation}\label{eq Friedman Russo contractive projection identity} P \{P(a), P (b), c \} = P \{P(a), P (b), P (c) \},
  \end{equation} holds for all $a,b,c\in E$ (see \cite[Theorem 3]{FriRu87conditional}). It should be noted that when $P(E)$ is already a JB$^*$-subtriple of $E$ (for example, when $P= P_j(e)$ for some tripotent $e\in E$ and $j\in \{0,1,2\}$), the right hand sides in \eqref{eq Kaup contractive projection identity} and \eqref{eq Friedman Russo contractive projection identity} can be replaced with $ \{P(a), P (b), P (c) \}$.

\begin{lemma}\label{l 1} Let $\delta: E\to E$ be a triple derivation on a JB$^*$-triple. Suppose $P: E\to E$ is a contractive projection and $P(E)$ is a subtriple of $E$. Then the mapping $D=  P \delta|_{P(E)} : P(E) \to P(E)$ is a triple derivation on the JB$^*$-triple $(P(E),\{.,.,.\}_{_P})=(P(E),\{.,.,.\})$.
\end{lemma}

\begin{proof} Given $a,b,c\in P(E)$ we have $$\begin{aligned} D\{a,b,c\}_{_P} &= P \delta \{a,b,c\} = P (\{\delta (a),b,c\}+ \{a,\delta (b),c\} + \{a,b,\delta (c)\})\\
  &= \{P \delta (a),b,c\}+ \{a,P\delta (b),c\} + \{a,b,P\delta (c)\} \\
  &= \{D (a),b,c\}+ \{a,D (b),c\} + \{a,b,D (c)\}\\
  &= \{D (a),b,c\}_{_P}+ \{a,D (b),c\}_{_P} + \{a,b,D (c)\}_{_P},
  \end{aligned}$$ where we have applied \eqref{eq Kaup contractive projection identity}, \eqref{eq Friedman Russo contractive projection identity} and the fact that $a,b,c\in P(E)$, the latter being a JB$^*$-subtriple of $E$.
\end{proof}

Suppose $\delta$ is a triple derivation on a JB$^*$-triple $E$. For each tripotent $e$ in $E$, the Peirce projection $P_j(e)$ is a contractive projection for every $j\in \{0,1,2\}$. It is further known that the projection $P_2(e) + P_0(e)$ is contractive too (see \cite[Corollary 1.2]{FriRu85}). Moreover, in this case, the corresponding images of these contractive projections, i.e. the Peirce subspaces $E_j(e) = P_j(e) (E),$ $j=0,1,2$ and $E_2(e)\oplus E_0(e)$, are JB$^*$-subtriples of $E$. It follows from Lemma \ref{l 1} that $D_j^e = P_j(e) \delta|_{E_j(e)} : E_j(e) \to E_j(e)$ and $D_{0,2}^e = (P_0(e) + P_2(e)) \delta|_{E_0(e)\oplus E_2(e)} : E_0(e)\oplus E_2(e) \to E_0(e)\oplus E_2(e)$ are triple derivations. \smallskip

What about the reciprocal implication? That is, suppose $T: E\to E$ is a (bounded) linear operator on a JB$^*$-triple satisfying that for each tripotent $e\in E$, the mapping $D_{0,2}^e = (P_0(e) + P_2(e)) T|_{E_0(e)\oplus E_2(e)} : E_0(e)\oplus E_2(e) \to E_0(e)\oplus E_2(e)$ is a triple derivation. Is $T$ a triple derivation?

\begin{proposition}\label{p first characterization of triple derivations} Let $T:E \to E$ be a bounded linear mapping on a JB$^*$-triple. Suppose that the set of tripotents is norm-total in $E$, that is, every element in $E$ can be approximated in norm by a finite linear combination of mutually orthogonal tripotents. Then the following statements are equivalent:\begin{enumerate}[$(a)$]\item $T$ is a triple derivation;
\item For each tripotent $e$ in $E$ the mapping $D_{0,2}^e = (P_0(e) + P_2(e)) \delta|_{E_0(e)\oplus E_2(e)} : E_0(e)\oplus E_2(e) \to E_0(e)\oplus E_2(e)$ is a triple derivation.
\item For each tripotent $e$ in $E$ the mapping $D_{2}^e=  P_2(e) T|_{E_2(e)} : E_2(e) \to E_2(e)$ is a triple derivation and $P_0(e) T(e)=0$.
\end{enumerate}
\end{proposition}

\begin{proof} The implication $(a)\Rightarrow (b)$ follows from Lemma \ref{l 1}.\smallskip

$(b)\Rightarrow (c)$ Let $e$ be a tripotent in $E$. By our hypotheses, $D_{0,2}^e$ is a triple derivation on $ E_0(e)\oplus E_2(e)$. It then follows that $$ D_{0,2}^e (e) = D_{0,2}^e \{e,e,e\} = 2 \{D_{0,2}^e(e),e,e\}+ \{e, D_{0,2}^e(e), e\} \in E_1(e)\oplus E_2(e),$$ where in the final step we applied Peirce arithmetic. Consequently, $$P_0(e) T(e)= P_0(e) D_{0,2}^e(e) =0.$$ Since $D_{0,2}^e(e)$ is a triple derivation on $E_0(e)\oplus E_2(e)$, a new application of Lemma \ref{l 1} implies that $P_2(e) T|_{E_2(e)} = P_2(e) D_{0,2}^e |_{E_2(e)} $ is a triple derivation on $E_2(e)$.\smallskip

$(c)\Rightarrow (a)$ By hypothesis, \begin{equation}\label{eq the Peirce 0 part vanishes} P_0(e) T(e)= 0,\end{equation} and the mapping $D_{2}^e=  P_2(e) T|_{E_2(e)}$ is a triple derivation. Therefore, $$P_2(e) T (e) = P_2(e) T\{e,e,e\} = 2 \{  P_2(e) T(e),e,e\} +  \{e,P_2(e) T(e),e\},$$ and thus $$\begin{aligned}
2\{T(e),e,e\} &+\{e,T(e),e\} = 2\{P_2(e) T(e),e,e\} + 2 \{P_1(e) T(e),e,e\} \\
&+\{e,P_2(e) T(e),e\}\\
& = P_1(e)T(e) + 2\{P_2(e) T(e),e,e\} + \{e,P_2(e) T(e),e\} \\
&= P_1(e) T(e) + P_2(e) T(e) = T(e).
\end{aligned}$$
We have therefore shown that \begin{equation}\label{eq T is a triple der on tripotents} T(e) = T\{e,e,e\} = 2 \{T(e),e,e\} + \{e,T(e),e\},
 \end{equation} for all tripotent $e\in E$. We can reproduce now an argument taken from the proof of \cite[Theorem 2.4]{BurFerPe2014}, to get the desired statement. If we consider a finite linear combination of mutually orthogonal tripotents $e_1,\ldots,e_n$, we deduce from the above that \begin{equation}\label{eq LHS at algebriac elements} T\left\{\sum_{j=1}^n \lambda_j e_j,\sum_{j=1}^n \lambda_j e_j, \sum_{j=1}^n \lambda_j e_j\right\}  =  \sum_{j=1}^n |\lambda_j|^2 \lambda_j T\{e_j,e_j,e_j\}
 \end{equation}
$$= \sum_{j=1}^n |\lambda_j|^2 \lambda_j (2 \{T(e_j),e_j,e_j\} + \{e_j,T(e_j),e_j\}).$$

On the other hand, if we fix three index $i,j,k\in \{1,\ldots,n\}$ with $i,k\neq j$, it follows from the fact that $e_i, e_k\in E_0 (e_j)$, \eqref{eq the Peirce 0 part vanishes} and Peirce arithmetic that \begin{equation}\label{eq both extremes orthogonal to the middle} \{e_i, T(e_j), e_k\} =0.
 \end{equation} Since $e_i \pm e_j$ is a tripotent, \eqref{eq T is a triple der on tripotents} implies that $$\begin{aligned}T(e_i) \pm T(e_j) &=  T(e_i \pm e_j) =  2 \{T(e_i \pm e_j),e_i \pm e_j,e_i \pm e_j\} \\
&+ \{e_i \pm e_j,T(e_i \pm e_j),e_i \pm e_j\}\\
& = 2 \{T(e_i),e_i ,e_i\} + 2 \{T(e_i), e_j, e_j\}\pm 2 \{T(e_j),e_i ,e_i\} \\
& \pm 2 \{T(e_j), e_j, e_j\} + \{e_i,T(e_i), e_i \} \pm \{e_i,T(e_i),  e_j\}\\
& + \{e_i ,T(e_j), e_j\} \pm \{e_j,T(e_i),e_i\} +\{e_j,T(e_j),e_i\} \pm \{e_j,T(e_j), e_j\}
\end{aligned}$$ which combined with \eqref{eq T is a triple der on tripotents} gives
\begin{align}\label{eq consequence 1 for two orthogonal trip} 0 &= 2 \{T(e_i), e_j, e_j\}\pm 2 \{T(e_j),e_i ,e_i\}  \pm \{e_i,T(e_i),  e_j\} + \{e_i ,T(e_j), e_j\} \\
&\pm \{e_j,T(e_i),e_i\} +\{e_j,T(e_j),e_i\} \nonumber,
\end{align} and then
\begin{align}\label{eq consequence 2 for two orthogonal trip} 0 &= \{T(e_i), e_j, e_j\} + \{e_i ,T(e_j), e_j\} .
\end{align}

Now, we check the following summands
\begin{equation}\label{eq RHS at algebriac elements} 2 \left\{T\left(\sum_{j=1}^n \lambda_j e_j\right),\sum_{j=1}^n \lambda_j e_j, \sum_{j=1}^n \lambda_j e_j\right\} + \left\{\sum_{j=1}^n \lambda_j e_j,T\left(\sum_{j=1}^n \lambda_j e_j\right), \sum_{j=1}^n \lambda_j e_j\right\}
\end{equation}
$$ = 2 \sum_{j=1, k=1}^n |\lambda_j|^2 \lambda_k \{T(e_k),e_j,e_j\} + \sum_{i=1, j=1, k=1}^n \lambda_i \overline{\lambda_j} \lambda_k \{e_i, T(e_j), e_k\} $$
$$=\hbox{(by \eqref{eq both extremes orthogonal to the middle})} = 2 \sum_{j=1, i=1}^n |\lambda_j|^2 \lambda_i \{T(e_i),e_j,e_j\} + \sum_{i=1, k=1, i\neq k}^n |\lambda_i|^2  \lambda_k \{e_i, T(e_i), e_k\} $$
$$ + \sum_{i=1, j=1, i\neq j}^n \lambda_i |\lambda_j|^2 \{e_i, T(e_j), e_j\} + \sum_{j=1}^n \lambda_j |\lambda_j|^2 \{e_j, T(e_j), e_j\}  $$
$$= \hbox{(by \eqref{eq consequence 2 for two orthogonal trip})} = 2 \sum_{j=1}^n |\lambda_j|^2 \lambda_j \{T(e_j),e_j,e_j\} + |\lambda_j|^2 \lambda_j \{T(e_j),e_j,e_j\}.$$

By combining \eqref{eq LHS at algebriac elements} and \eqref{eq RHS at algebriac elements} it ca be concluded that $$T\{a,a,a\} = 2 \{T(a),a,a\} + \{a, T(a),a\},$$ for every $a = \sum_{j=1}^n \lambda_j e_j$, where $e_1,\ldots, e_n$ are mutually orthogonal tripotents in $E$. Since, by hypotheses, tripotents in $E$ are norm-total, we get  $T\{a,a,a\} = 2 \{T(a),a,a\} + \{a, T(a),a\},$ for every $a\in E$. A standard polarization identity proves that $T$ is a triple derivation.
\end{proof}

\begin{remark}\label{r compact and JBWtriples satisfy the previous hypotheses}{\rm It should be noted here that every JBW$^*$-triple and every compact JB$^*$-triple satisfies the hypotheses of the previous proposition (see \cite[Lemma 3.11]{Horn87} and \cite{BuChu92}).}
\end{remark}

The novelty here is that, as we shall see in subsequent results, Lemmas \ref{l 1} and \ref{l 1 for linear maps triple derivable at orthogonal pairs} give a hint to study linear maps which are triple derivable at orthogonal pairs in terms of generalized Jordan derivations.

\begin{lemma}\label{l 1 for linear maps triple derivable at orthogonal pairs} Let $T: E\to E$ be a linear map on a JB$^*$-triple which is triple derivable at orthogonal pairs. Then for each contractive projection $P$ on $E$ with $P(E)$ being a subtriple of $E$, the mapping $PT|_{P(E)}: P(E)\to P(E)$ is triple derivable at orthogonal pairs. In particular, for each tripotent $e\in E$, the mappings $P_j(e) T|_{E_j(e)}: E_j(e) \to E_j (e)$ {\rm(}$j=0,1,2${\rm)}  and $(P_0(e)+P_2(e)) T|_{E_0(e)\oplus E_2(e)}: E_0(e)\oplus E_2(e) \to E_0(e)\oplus E_2(e)$ are triple derivable at orthogonal pairs.
\end{lemma}

\begin{proof} The ideas are very similar to those given in the proof of Lemma \ref{l 1}. Given $a,b,c\in P(E)$ with $a\perp b$ we have $$\begin{aligned} 0=P T\{a,b,c\}_{_P} &= P T \{a,b,c\} = P (\{T (a),b,c\}+ \{a,T (b),c\} + \{a,b,T (c)\}\\
&= \{P T (a),b,c\}+ \{a,PT (b),c\} + \{a,b,PT (c)\}\\
  &= \{P T (a),b,c\}_{_P}+ \{a,PT (b),c\}_{_P} + \{a,b,PT (c)\}_{_P}.
  \end{aligned}$$ where we have applied \eqref{eq Kaup contractive projection identity}, \eqref{eq Friedman Russo contractive projection identity}, the fact that $a,b,c\in P(E)$, and the assumption that $P(E)$ is a JB$^*$-subtriple of $E$.
\end{proof}

\section{Generalized Jordan derivations via orthogonal forms}

Lemma \ref{l 1 for linear maps triple derivable at orthogonal pairs} justifies the importance of determining the structure of those bounded linear maps on a JB$^*$-algebra which are triple derivable at orthogonal pairs. The Jordan analogues of the results in \cite[\S 3]{BurFerPe2014} and \cite[\S 2]{AyuKudPe2014} remain unexplored. The aim of this section is to complete our knowledge on Jordan and triple derivations by studying those linear maps on JB$^*$-algebras preserving the natural relations with respect to orthogonality that triple derivations enjoy.\smallskip

As we have already mentioned, given a triple derivation $\delta$ on a JB$^*$-triple $E$ and elements $a,b,c\in E$ with $a,c\perp b$ we have $$0 = \delta\{a,b,c\} = \{\delta(a), b,c\} + \{ a, \delta(b), c\} +\{ a,b,\delta (c)\} =  \{ a, \delta(b), c\}.$$ Similarly, if $T:E\to E$ is a linear mapping which is triple derivable at orthogonal elements we deduce that $T$ satisfies the following property \begin{equation}\label{eq orthogonal property}\hbox{$\{ a, T(b), c\} =0,$ for all $a,b,c\in E$ with $a,c\perp b$.}\end{equation} So, as in the case of bounded linear maps on C$^*$-algebras, it seems interesting to study those bounded linear maps on $E$ satisfying the just commented property \eqref{eq orthogonal property}. We note that property \eqref{eq orthogonal property} is related to the property studied in \cite[Lemma 2.8]{EssPe2018} for linear maps triple derivable at zero.\smallskip

As in the associative case, for each mapping $T$ on a JB$^*$-algebra $A$, we define a mapping $T^{\sharp} : A\to A$ given by $T^{\sharp} (x) := T(x^*)^*$. Clearly $T$ is (bounded) linear if and only if the same property holds for $T^{\sharp}$. The mapping $T$ will be called symmetric (respectively, anti-symmetric) if $T^{\sharp} = T$ (respectively, $T^{\sharp} = -T$). Every mapping $T$ can be written as the sum of a symmetric and an anti-symmetric mapping and a linear combination of two symmetric mappings, $T = \frac12 (T^{\sharp} + T) + \frac12 (T-T^{\sharp}) =  \frac12 (T^{\sharp} + T) + i  \frac{1}{2 i} (T-T^{\sharp})$.\smallskip

\begin{lemma}\label{l properies of the conjugate} Let $T$ be a linear mapping on a JB$^*$-algebra. The following statements hold:
\begin{enumerate}[$(a)$]
\item $T$ is triple derivable at orthogonal elements if and only if $T^{\sharp}$ satisfies the same property;
\item The real linear combination of linear mappings which are triple derivable at orthogonal elements also satisfies this property;
\item $T$ is triple derivable at orthogonal elements if and only if $\frac12 (T^{\sharp} + T)$ and $\frac12 (T-T^{\sharp})$ are triple derivable at orthogonal elements.
\end{enumerate}
\end{lemma}

\begin{proof} $(a)$ Suppose $T$ is triple derivable at orthogonal elements. Take $a,b,c\in A$ with $a\perp b$. Since $a^*\perp b^*$, it follows from the hypothesis that $$0= \{T(a^*), b^*, c^*\}+ \{a^*, T(b^*), c^*\}+ \{a^*, b^*, T(c^*)\}.$$ By applying the involution on $A$ we obtain $$0= \{T(a^*)^*, b, c\}+ \{a, T(b^*)^*, c\}+ \{a, b, T(c^*)^*\}$$ $$= \{T^{\sharp}(a), b, c\}+ \{a, T^{\sharp}(b), c\}+ \{a, b, T^{\sharp}(c)\}.$$ Since $T^{\sharp\sharp} =T$, the ``if'' implication also follows from the ``only if'' one.\smallskip

$(b)$ can be straightforwardly checked, and $(c)$ follows from $(a)$ and $(b)$.

%
\end{proof}

Let $G: A\to X$ be a linear mapping where $A$ is a JB$^*$-algebra and $X$ is a Jordan-Banach $A$-module. Following \cite[\S 4]{AlBreExVill09} and \cite[\S 3]{BurFerPe2014}, we shall say that $G$ is a \emph{generalized Jordan derivation} if there exists $\xi\in X^{**}$ satisfying \begin{equation}\label{eq generalized Jordan derivation} G (a\circ b) = G(a)\circ b + a\circ G(b) - U_{a,b} (\xi ), \end{equation} for every $a,b$ in $A$. If $A$ is unital the element $\xi= G(1)$ lies in $X$. An example can be given as follows, fix a non-zero $a$ in $A$. The mapping $M_a$ is a generalized Jordan derivation on $A$ which is not a Jordan derivation because $M_a (1) = a \neq 0$. If $a\neq -a^*$, the mapping $M_a$ neither is a triple derivation because $M_a(1) = a \neq -a^* = -M_a(1)^*$ (cf. \cite[Lemma 1 and its proof]{HoMarPeRu}).\smallskip

\begin{remark}\label{r module is the dual}{\rm In our case, we shall be mainly interested in generalized Jordan derivations $G: A\to A^*$, where the Jordan module operation in $A^*$ is given by $(\varphi \circ a) (b) =\varphi (a\circ b)$ ($a,b\in A$). Suppose $G:A\to A^*$ is a generalized derivation (which is automatically continuous by \cite[Proposition 2.1]{JamPeSidd2015}) for some $\xi\in A^{***}$. Let $(c_{\lambda})_{\lambda}$ be a bounded increasing approximate unit in $A$ (see \cite[Proposition 3.5.4]{HOS}). The net $(G (a\circ c_{\lambda}))_{\lambda}\to G(a)$ in norm, and $(G(a)\circ c_{\lambda})_{\lambda}\to G(a)$ in the weak$^*$ topology of $A^*$. The net $(G(c_{\lambda}))_{\lambda}$ is bounded in a dual Banach space, and hence we can find a subnet (denoted by the same symbol) converging in the weak$^*$ topology of $A^*$ to some $\phi\in A^*$. Therefore, $(a\circ G(c_{\lambda}))_{\lambda}\to a\circ \phi$ in the weak$^*$ topology. On the other hand, for each $d\in A$, the net $U_{a,c_{\lambda}} (\xi ) (d) = \xi (U_{a,c_{\lambda}} (d))\to \xi (a\circ d)$, as $\|U_{a,c_{\lambda}} (d)-a\circ d\|\to 0$. Finally, by combining all these facts with \eqref{eq generalized Jordan derivation} we get $$ G(a) (d) = G(a) (d) + \phi (a\circ d) -  \xi (a\circ d),$$ for all $a,d\in A$, which clearly guarantees that we can take $\xi =\phi \in A^*$.}
\end{remark}

A \emph{purely Jordan generalized derivation} from $A$ into $A^*$ is a generalized Jordan derivation $G: A \to A^*$ satisfying $G(a) (b) = G(b) (a)$, for every $a,b\in A$ (cf. \cite[Definition 3.4]{JamPeSidd2015}), while a Jordan derivation $D$ from $A$ into $A^*$ is said to be a \emph{Lie Jordan derivation} if $D(a) (b) =- D(b) (a)$, for every $a,b\in A$ (cf. \cite[Definition 3.12]{JamPeSidd2015}).

\smallskip

The results in \cite{JamPeSidd2015} reveal the strong connections between generalized Jordan derivations and orthogonal forms on JB$^*$-algebras. Let $A$ be a JB$^*$-algebra. We recall that a continuous bilinear form $V : A\times A \to \mathbb{C}$ is called \emph{orthogonal} if $V (a,b^*) = 0,$ for every $a,b\in A$ with $a\perp b$. If $V (a,b) = 0$ only for elements $a,b\in A_{sa}$ with $a\perp b$, we say that $V$ is \emph{orthogonal on $A_{sa}$} (cf. \cite{JamPeSidd2015}). Corollary 3.14 (see also Propositions 3.8 and 3.9) in \cite{JamPeSidd2015} proves that a bilinear form on a JB$^*$-algebra ${A}$ is orthogonal if and only if it is orthogonal on $A_{sa}.$ A bilinear form $V$ on $A$ is called symmetric (respectively, anti-symmetric) if $V(a,b) = V(b,a)$ (respectively, $V(a,b) = -V(b,a)$) for all $a,b\in A$.\smallskip

Let $\mathcal{OF}_{s}(A)$ (respectively, $\mathcal{OF}_{as}(A)$) denote the space of all symmetric orthogonal forms on $A$ (respectively, of all anti-symmetric orthogonal forms on $A$), and let $\mathcal{PJGD}er(A,A^*)$ and $\mathcal{L}ie\mathcal{JD}er(A,A^*)$ stand for the spaces of all purely Jordan generalized derivations from $A$ into $A^*$ and of all Lie Jordan derivations from $A$ into $A^*$, respectively.\smallskip

For each $V\in \mathcal{OF}_{s}(A)$ define $G_{_V} : A \to A^*$ in $\mathcal{PJGD}er(A,A^*)$ given by $G_{_V} (a) (b) = V(a,b)$, and for each $G\in \mathcal{PJGD}er(A,A^*)$ we set $V_{_G} : A\times A \to C$, $V_{_G} (a,b) := G(a) (b)$ $(a,b\in A)$.
By \cite[Theorem 3.6]{JamPeSidd2015}, the mappings $$\mathcal{OF}_{s}(A) \to \mathcal{PJGD}\hbox{er}(A,A^*), \ \ \mathcal{PJGD}\hbox{er}(A,A^*) \to \mathcal{OF}_{s}(A),$$ $$V\mapsto G_{_V}, \ \ \ \  \ \ \ \  \ \ \ \  \ \ \ \  \ \ \ \  \ \ \ \  \ \ \ \ \ \ \ \  \ \ \ \ G\mapsto V_{_G},$$ define two (isometric) linear bijections which are inverses of each other.\smallskip

For each $V\in \mathcal{OF}_{as}(A)$ we define $D_{_V} : A \to A^*$ in $\mathcal{L}ie\mathcal{JD}er(A,A^*)$ given by $D_{_V} (a) (b) = V(a,b)$, and for each $D\in \mathcal{L}ie\mathcal{JD}er(A,A^*)$ we set $V_{_D} : A\times A \to \mathbb{C}$, $V_{_D} (a,b) := D(a) (b)$ $(a,b\in A)$.
By \cite[Theorem 3.13]{JamPeSidd2015}, the mappings
$$\mathcal{OF}_{as}(A) \to \mathcal{L}\hbox{ie}\mathcal{JD}\hbox{er}(A,A^*), \ \ \mathcal{L}\hbox{ie}\mathcal{JD}\hbox{er}(A,A^*) \to \mathcal{OF}_{as}(A),$$ $$V\mapsto D_{_V}, \ \ \ \  \ \ \ \  \ \ \ \  \ \ \ \  \ \ \ \  \ \ \ \  \ \ \ \ \ \ \ \  \ \ \ \ D\mapsto V_{_D},$$ define two linear bijections which are inverses of each other.\smallskip

We deal first with symmetric bounded linear maps which are triple derivable at orthogonal pairs.

\begin{proposition}\label{p symmetric operators} Let $T: A\to A$ be a symmetric bounded linear operator on a JB$^*$-algebra. Suppose that $T$ is triple derivable at orthogonal pairs. Then $T$ is a generalized Jordan derivation. Furthermore, there exist a central element $\xi\in A^{**}_{sa}$ and a Jordan $^*$-derivation $D: A\to A^{**}$ such that $T(a) = D(a) + M_{\xi} (a) = D(a) + \xi \circ a,$ for all $a\in A$. Clearly $D$ is a triple derivation.
\end{proposition}

\begin{proof} Pick an arbitrary functional $\phi\in A^*$. We consider the following symmetric bounded bilinear form $V: A\times A\to \mathbb{C}$, $V(a,b) := \phi(T(a)\circ b + a\circ T(b))$. Let $(c_{\lambda})_{\lambda}$ be a bounded increasing approximate unit in $A$ (cf. \cite[Proposition 3.5.4]{HOS}). Given $a,b\in A_{sa}$ with $a\perp b$, the hypotheses assure that $$\begin{aligned}0= T\{a,b,c_{\lambda}\} &= \{T(a), b, c_{\lambda}\} + \{a, T(b) ,c_{\lambda}\} + \{a,b,T(c_{\lambda})\}\\
& = \{T(a), b, c_{\lambda}\} + \{a, T(b) ,c_{\lambda}\}, \hbox{ for all } \lambda.
\end{aligned} $$ Taking norm limits in $\lambda$ we arrive at $$0 = \{T(a), b, 1\} + \{a, T(b) ,1\}= T(a) \circ b + a\circ T(b)^*.$$ By applying that $T$ is symmetric and $b\in A_{sa}$ we get $0 = T(a) \circ b + a\circ T(b),$ and thus $V(a,b) =0$. We have therefore shown that $V$ is orthogonal on $A_{sa}$ and hence on the whole $A$ (see \cite[Corollary 3.14]{JamPeSidd2015}). Theorem 3.6 in \cite{JamPeSidd2015} implies that the mapping $G_V : A\to A^*$, $G_V (a) (b) = V(a,b)$ is a purely Jordan generalized derivation, that is, by Remark \ref{r module is the dual} there exists $\varphi \in A^*$ such that
$$ G_V (a \circ b ) = G_V(a) \circ b + a\circ G_{V} (b) - U_{a,b} (\varphi),$$ for all $a,b\in A$, or equivalently, $$ G_V (a \circ b ) (c) = (G_V(a) \circ b + a\circ G_{V} (b) - U_{a,b} (\varphi)) (c),$$ $$ V (a \circ b  , c) = V(a, c\circ b) +  V (b,a\circ c) - \varphi U_{a,b} (c),$$ or
$$ \phi(T(a \circ b)\circ c + (a \circ b)\circ T(c) ) =  \phi(T(a)\circ (c\circ b) + a\circ T(c\circ b)) $$
$$+  \phi(T(b)\circ (a\circ c) + b\circ T(a\circ c)) - \varphi U_{a,b} (c),$$
for all $a,b,c\in A.$ If we replace $c$ with the elements in a bounded increasing approximate unit $(c_{\lambda})_{\lambda}$ in $A$, and we take norm and weak$^*$ limits in $\lambda$ we arrive at \begin{equation}\label{eq penultimate prop symmetric} \phi( T(a \circ b) + (a \circ b)\circ T^{**}(1) ) = \phi(2 T(a)\circ  b + 2 a\circ T(b) ) - \varphi (a\circ b),
\end{equation}
for all $a,b\in A,$ where we have applied the well known fact that $(c_{\lambda})_{\lambda}\to 1$ in the weak$^*$ topology of $A^{**}$. By replacing $b$ with $c_{\lambda}$ in the above identity and taking norm and weak$^*$ limits we are led to the identity $$ \phi( T(a) + a \circ T^{**}(1) ) = \phi(2 T(a) + 2 a \circ T^{**}(1) ) - \varphi (a),$$ or equivalently, $$ \varphi (a) = \phi( T(a) +  a \circ T^{**}(1) )$$
for all $a\in A.$ Now, by substituting this latter identity in \eqref{eq penultimate prop symmetric} it follows that $$\phi( T(a \circ b) + (a \circ b)\circ T^{**}(1) ) = \phi(2 T(a)\circ  b + 2 a\circ T(b) ) - \phi( T(a\circ b) +  (a\circ b) \circ T^{**}(1) ),$$ which simplified gives
$$ \phi T(a \circ b)  = \phi( T(a)\circ  b +  a\circ T(b)- (a\circ b) \circ T^{**}(1) ),$$ for all $a,b\in A$. The arbitrariness of $\phi$ combined with the Hahn-Banach theorem prove that \begin{equation}\label{eq fundamental identity 1 in Propo 1} T(a \circ b)  = T(a)\circ  b +  a\circ T(b)- (a\circ b) \circ T^{**}(1),
 \end{equation} for all $a,b\in A$. Since $T$ is symmetric, $T^{**}$ enjoys the same property, and hence $T^{**} (1)\in A_{sa}^{**}.$ We can actually deduce from the the weak$^*$-density of $A$ in $A^{**}$ (Goldstine's Theorem), the separate weak$^*$-continuity of the product of $A^{**}$ (cf. \cite[Theorem 4.4.16 and Corollary 4.1.6]{HOS}) and the weak$^*$ continuity of $T^{**}$ that the identity in \eqref{eq fundamental identity 1 in Propo 1} holds for all $a,b\in A^{**}$ by just replacing $T$ with $T^{**}$.\smallskip

We shall next prove that $\xi= T^{**} (1)$ lies in the centre of $A^{**}$. By the conclusions above, for each projection $p\in A^{**}$ we have \begin{equation}\label{eq key identity prop 1 projections} T^{**}(p) = 2 T^{**}(p) \circ p - p \circ \xi.
 \end{equation} Let $T^{**}(p) = x_2 +x_1 +x_0$ and $\xi = y_2 +y_1 +y_0$ denote the Peirce decompositions of $T^{**}(p)$ and $\xi$ with respect to $p$, respectively (i.e., $x_j = P_j(p) (T^{**}(p))$ and $y_j = P_j(p) (\xi)$ for $j=0,1,2$). It is easy to check that $x_2 + \frac12 x_1=\{p,p,T^{**}(p)\} = p\circ T^{**}(p)$ and $y_2 + \frac12 y_1=\{p,p,\xi\} = p\circ \xi.$ Then the previous identity \eqref{eq key identity prop 1 projections} gives $$ x_2 +x_1 +x_0 =  T^{**}(p) = 2 T^{**}(p) \circ p - p \circ \xi = 2 x_2 + x_1 - y_2 - \frac12 y_1,$$ which guarantees that $y_1=0$. Now, having in mind that $\xi,p\in A^{**}_{sa}$ we deduce that $P_2(p) (\xi) = \{p,\{p,\xi,p\},p\} = U_p (U_p (\xi^*)^*) = U_p^2 (\xi) = U_p(\xi)$ (see \cite[\S 2.6]{HOS}), and thus $U_p(\xi) =  P_2(p) (\xi) = y_2$. On the other hand $p^2 \circ \xi = p\circ \xi = \{p,p,\xi\} = y_2 +\frac12 y_1 = y_2.$ Therefore $p^2 \circ \xi = U_p(\xi)$, which guarantees that $p$ and $\xi$ operator commute (cf. comments in page \pageref{eq operator commutativity}). Since projections in $A^{**}_{sa}$ are norm-total (cf. \cite[Proposition 4.2.3]{HOS}), we conclude that $\xi$ lies in the centre of $A^{**}$.\smallskip

Finally, by applying that $\xi= T^{**} (1)$ is a central element in $A^{**}$, we deduce from \eqref{eq fundamental identity 1 in Propo 1} that the identity $$T(a \circ b)  = T(a)\circ  b +  a\circ T(b)- U_{a, b}  (\xi),$$ holds for all $a,b\in A$, witnessing that $T$ is a generalized Jordan derivation. It is easy to check that the mapping $D: A\to A^{**}$, $D (a) = T(a) - \xi \circ a$ is a Jordan $^*$-derivation, and $T$ satisfies $T(a) = D(a) + M_{\xi} (a)$ for all $a\in A$.
\end{proof}

The next corollary seems to be a new advance too.

\begin{corollary}\label{c symmetric operators Cstaralgebras} Let $T: A\to A$ be a symmetric bounded linear operator on a C$^*$-algebra. Suppose that $T$ is triple derivable at orthogonal pairs. Then $T$ is a generalized derivation. Furthermore, there exist a central element $\xi\in A^{**}_{sa}$ and a $^*$-derivation $D: A\to A^{**}$ such that $T(a) = D(a) + M_{\xi} (a) = D(a) + \xi a,$ for all $a\in A$.
\end{corollary}

\begin{proof} Proposition \ref{p symmetric operators} proves the existence of a central element $\xi\in A^{**}_{sa}$ and a Jordan $^*$-derivation $D: A\to A^{**}$ such that $T(a) = D(a) + M_{\xi} (a) = D(a) + \xi \circ a,$ for all $a\in A$. A celebrated result of B.E. Johnson asserts that every Jordan derivation on a C$^*$-algebra is a derivation (cf. \cite[Theorem 6.3]{John1996}). Therefore, $D$ is a $^*$-derivation. The rest is clear because $\xi$ is a central element.
\end{proof}

Let $A$ be a JB$^*$-algebra. Following \cite{Ed80}, the \emph{(Jordan) multipliers algebra} of $A$ is the set $$M(A):=\{x\in A^{**}: x\circ A \subseteq  A\}.$$  The space $M(A)$ is a unital JB$^*$-subalgebra of $A^{**}$. Moreover, $M(A)$ is the (Jordan) idealizer of $A$ in $A^{**}$, that is, the largest JB$^*$-subalgebra of $A^{**}$ containing $A$ as a closed Jordan ideal (cf. \cite[Theorem 2]{Ed80}).\smallskip

We deal next with the anti-symmetric operators which are triple derivable at orthogonal pairs.

\begin{proposition}\label{p anti-symmetric operators unital} Let $T: A\to A$ be an anti-symmetric bounded linear operator on a JB$^*$-algebra. Suppose that $T$ is triple derivable at orthogonal pairs. Then $T$ is a triple derivation, moreover, there exists an anti-symmetric element $\eta$ in the multiplier algebra of $A$, such that $T(x) = \eta \circ x$ for all $x\in A$.
\end{proposition}

\begin{proof} Fix an arbitrary $\phi\in A^*$. Let us consider the following continuous bilinear form $V(a,b) = \phi(- a\circ T(b) + T(a) \circ b)$ ($a,b\in A$). Clearly $V$ is anti-symmetric. Let us take $a,b\in A_{sa}$ with $a\perp b$. By hypothesis, $0=T\{a,b,1\} = \{T(a),b,1\} + \{a,T(b),1\} + \{a,b,T(1)\} = \{T(a),b,1\} + \{a,T(b),1\},$ and thus
$$0=\{T(a),b,1\} + \{a,T(b),1\} = T(a) \circ b + a\circ T(b)^* = T(a) \circ b - a\circ T(b),$$ where in the last equality we applied that $T^{\sharp} = -T$. If $A$ is not unital, we choose a bounded increasing approximate unit $(c_{\lambda})_{\lambda}$ in $A$ (see \cite[Proposition 3.5.4]{HOS}), and the hypotheses give $$\begin{aligned}0=T\{a,b,c_{\lambda}\} &= \{T(a),b,c_{\lambda}\} + \{a,T(b),c_{\lambda}\} + \{a,b,T(c_{\lambda})\} \\
&= \{T(a),b,c_{\lambda}\} + \{a,T(b),c_{\lambda}\} \hbox{ for all } \lambda.
\end{aligned}$$ Taking norm limits in $\lambda$ we obtain $$0= T(a) \circ b + a\circ T(b)^* = T(a) \circ b - a\circ T(b),$$ where, as before, we applied that $T^{\sharp} = -T$. We have therefore shown that $V(a,b) = \phi( T(a) \circ b - T(b) \circ a)  = 0,$ and thus $V$ is orthogonal on $A_{sa}$, and hence on $A$ (cf. \cite[Corollary 3.14]{JamPeSidd2015}). It follows from \cite[Theorem 3.13]{JamPeSidd2015} that the mapping $D_{V} : A\to A^*,$ $D_V(a) (b) = V(a,b)$ is a Lie Jordan derivation. Consequently, $$D_V(a\circ c) = D_V(a) \circ c + a\circ D_V( c), \hbox{ for all } a,c\in A, $$ equivalently,
$$V(a\circ c, b) =  D_V(a\circ c) (b) = (D_V(a) \circ c + a\circ D_V( c)) (b) = V(a, b\circ c) +V(c, a\circ b), $$ and
$$\phi( T(a\circ c) \circ b -  T(b) \circ (a\circ c)) = \phi( T(a) \circ (b\circ c) - T(b\circ c) \circ a  +T(c) \circ (a\circ b) - T(a\circ b) \circ c),$$
for all $a,b,c\in A$. Now, having in mind the arbitrariness of $\phi$, we deduce from the Hahn-Banach theorem that
\begin{equation}\label{eq base identity in Prop 2}\begin{split}
T(a\circ c) \circ b -  T(b) \circ (a\circ c) & = T(a) \circ (b\circ c) - T(b\circ c) \circ a  \\
&+T(c) \circ (a\circ b) - T(a\circ b) \circ c,
\end{split} \end{equation} for all $a,b,c\in A$.\smallskip

Let us assume that $A$ is unital. In this case, by replacing $c=1$ in \eqref{eq base identity in Prop 2} we obtain $$- a\circ T(b) +  T(a) \circ b = T(a) \circ b - T(b) \circ a  +T(1) \circ (a\circ b) - T(a\circ b) ,$$ for all $a,b\in A$, which implies that $ T(a\circ b) = (a\circ b)\circ T(1)$ for all $a,b\in A$.\smallskip

In the non-unital case, we can consider a bounded increasing approximate unit $(c_{\lambda})_{\lambda}$ in $A$ (see \cite[Proposition 3.5.4]{HOS}). Having in mind that $(c_{\lambda})_{\lambda}\to 1$ in the weak$^*$ topology of $A^{**}$, by replacing $c$ with $c_{\lambda}$ in \eqref{eq base identity in Prop 2} and taking weak$^*$ limits we get $$ - a\circ T(b) +  T(a) \circ b = T(a) \circ b - T(b) \circ a  +T^{**}(1) \circ (a\circ b) - T(a\circ b) ,$$ for all $a,b\in A$, and thus $T(a) = T^{**} (1) \circ a$ for all $a\in A$. Since $A\ni T(a) = T^{**} (1) \circ a$ for all $a\in A$, we conclude that $T^{**} (1)$ lies in the multiplier algebra of $A$. Finally, since $T^{**}$ is anti-symmetric too, the element $\eta =T^{**} (1)$ is anti-symmetric and satisfies the desired statement.
\end{proof}

\begin{remark}\label{r weaker hypotheses in Propositions 1 and 2 in the unital case}{\rm
Let us observe that in the case of unital JB$^*$-algebras, the hypothesis $T$ being triple derivable at orthogonal elements in Propositions \ref{p symmetric operators} and \ref{p anti-symmetric operators unital} can be reduced to the property $$\{T(a),b,1\} + \{a,T(b),1\}=0, \hbox{ for all } a,b\in A_{sa} \hbox{ with } a \perp b.$$ It can be checked that the proofs given above remain valid under this weaker assumption.
}\end{remark}

The C$^*$- version of Proposition \ref{p anti-symmetric operators unital} also is a new result worth to be stated.

\begin{corollary}\label{c anti-symmetric operators unital Cstar algebras} Let $T: A\to A$ be an anti-symmetric bounded linear operator on a C$^*$-algebra. Suppose that $T$ is triple derivable at orthogonal pairs. Then $T$ is a triple derivation, moreover, there exists an anti-symmetric element $\eta$ in the multiplier algebra of $A$, such that $T(x) = \eta \circ x$ for all $x\in A$.
\end{corollary}

We are now in a position to address the characterization of those bounded linear maps on a JB$^*$-algebra which are triple derivable at orthogonal elements.

\begin{theorem}\label{t bd linear maps triple derivable at orth pairs on a JBstar algebra} Let $T: A\to A$ be a bounded linear maps on a JB$^*$-algebra. Suppose $T$ is triple derivable at orthogonal elements. Then there exists a Jordan $^*$-derivation $D:A\to A^{**}$, a central element $\xi\in A^{**}_{sa},$ and an anti-symmetric element $\eta$ in the multiplier algebra of $A$, such that $$ T(a) = D(a) + \xi \circ a + \eta \circ a, \hbox{ for all } a\in A.$$ Moreover, the mapping $\delta: A\to A^{**}$, $\delta (a) = D(a) + \eta \circ a$ {\rm(}$a\in A${\rm)} is a triple derivation and $T(a) = \delta(a) + \xi\circ a,$ for all $a\in A$. The mapping $S: A\to A^{**}$, $S(a) = \xi \circ a$ is the restriction to $A$ of an element in the centroid of $A^{**}$.
\end{theorem}

\begin{proof} Let us write $T = T_1 + T_2$, where $T_1 = \frac{T + T^{\sharp}}{2}$ is symmetric and $T_2 = \frac{T - T^{\sharp}}{2}$ is anti-symmetric. Lemma \ref{l properies of the conjugate} assures that $T_1$ and $T_2$ are triple derivable at orthogonal pairs. Propositions \ref{p symmetric operators} and \ref{p anti-symmetric operators unital} applied to $T_1$ and $T_2$ prove the desired conclusion.
\end{proof}

\begin{remark}\label{r unital case for Theorem for JBstar algebras}{\rm If in Theorem \ref{t bd linear maps triple derivable at orth pairs on a JBstar algebra} the JB$^*$-algebra $A$ is unital, the Jordan $^*$-derivation $D$ and the triple derivation $\delta$ are $A$-valued, and $\xi$ is a symmetric element in the centre of $A$.}
\end{remark}

Theorem \ref{t bd linear maps triple derivable at orth pairs on a JBstar algebra} particularly holds when $A$ is a C$^*$-algebra.\smallskip

We can now complete the conclusion on linear maps which are triple derivable at zero.

\begin{corollary}\label{c tder at zero and tder at op equivalent on JBstar algebras} Let $T: A\to A$ be a bounded linear maps on a JB$^*$-algebra {\rm(}and in particular on a C$^*$-algebra{\rm)}. The the following assertions are equivalent:\begin{enumerate}[$(a)$]\item $T$ is triple derivable at zero;
\item $T$ is triple derivable at orthogonal elements;
\item There exists a Jordan $^*$-derivation $D:A\to A^{**}$, a central element $\xi\in A^{**}_{sa},$ and an anti-symmetric element $\eta$ in the multiplier algebra of $A$, such that $$ T(a) = D(a) + \xi \circ a + \eta \circ a, \hbox{ for all } a\in A;$$
\item There exist a triple derivation $\delta: A\to A^{**}$ and a symmetric element $S$ in the centroid of $A^{**}$ such that $T= \delta +S$.
\end{enumerate}
\end{corollary} 

\begin{proof} We have already commented in the introduction that $(a)\Rightarrow (b)$. The implications $(b)\Rightarrow (c) \Rightarrow (d)$ follow from Theorem \ref{t bd linear maps triple derivable at orth pairs on a JBstar algebra}. Finally, if $(d)$ holds and we take $a,b,c\in A$ with $\{a,b,c\} =0$. By the assumptions $$\{T(a), b,c \} + \{a, T(b), c\} + \{a,b,T(c)\} $$ $$=  \delta \{a, b,c \} + \{S(a), b,c \} + \{a, S(b), c\} + \{a,b,S(c)\}$$ $$ = S\{a,b,c\} + R\{a,b,c\} + S\{a,b,c\} = 0,$$ where $R$ is the corresponding element in the centroid of $A^{**}$ satisfying $ \{x, S(y), z\} = R\{x,y,z\}$ for all $x,y,z\in A^{**}$. 
\end{proof}

\section{A new characterization of triple derivations}

This section is devoted to present a new characterization of those continuous linear maps on a JBW$^*$-triple which are triple derivations in terms of a good local behavior on Peirce 2-subspaces. Let us begin with a property which was already implicit in \cite{HoMarPeRu} and in many other recent studies on triple derivations and local triple derivations (see for example \cite{BurFerGarPe2014,BurFerPe2014}). Let $\delta: E\to E$ be a triple derivation. Given a tripotent $e\in E$, the identity $$\delta(e) = \delta\{e,e,e\} = 2 \{\delta(e),e,e\} + \{e,\delta(e),e\},$$ combined with Peirce arithmetic show that $P_0(e) \delta(e) =0$ and $Q(e)\delta(e) = Q(e)P_2(e)\delta(e) = -P_2(e)\delta(e)$. The later implies that $P_2(e)\delta(e)$ is a skew symmetric element in the JB$^*$-algebra $E_2(e)$ (i.e. $(P_2(e)\delta(e))^{*_e} = -P_2(e)\delta(e)$). This is a necessary condition to be a triple derivation.

\begin{theorem}\label{t characterization of triple derivations} Let $T:E\to E$ be a bounded linear map, where $E$ is a JB$^*$-triple in which tripotents are norm-total. The following statements are equivalent:
\begin{enumerate}[$(a)$] \item $T$ is a triple derivation;
\item $T$ is triple derivable at orthogonal elements and for each tripotent $e$ in $E$, the element $P_2(e)T(e)$ is a skew symmetric element in the JB$^*$-algebra $E_2(e)$ {\rm(}i.e. $(P_2(e)T(e))^{*_e} = -P_2(e) T(e)${\rm)}.
\end{enumerate}
\end{theorem}

\begin{proof} The implication $(a)\Rightarrow (b)$ has been already commented.\smallskip

$(b)\Rightarrow (a)$ Let us fix a tripotent $e\in E$. The mapping $T$ is triple derivable at orthogonal elements, and hence $T$ satisfies the property in \eqref{eq orthogonal property}. Since $P_0(e) T(e)\perp P_2(e)(e)=e$ (cf. Peirce arithmetic in page \pageref{Peirce arithmetic}) it follows that $$0=\{ P_0(e) T(e), T(e), P_0(e) T(e)\}= \{ P_0(e) T(e), P_0 (e) T(e), P_0(e) T(e)\},$$ and consequently $P_0(e) T(e)= 0$ by the axioms of JB$^*$-triples.\smallskip

Lemma \ref{l 1 for linear maps triple derivable at orthogonal pairs} implies that the mapping $P_2(e) T|_{E_2(e)}: E_2(e) \to E_2(e)$ is triple derivable at orthogonal elements. Since $E$ is a (unital) JB$^*$-algebra, Theorem \ref{t bd linear maps triple derivable at orth pairs on a JBstar algebra} (see also Remark \ref{r unital case for Theorem for JBstar algebras}) assures the existence of Jordan $^*$-derivation $D_e:E_2(e)\to E_2(e)$, a central element $\xi_e\in E_2(e)_{sa},$ and an anti-symmetric element $\eta_e$ in the multiplier algebra of $E_2(e)$, such that $$P_2(e) T(a) = D_e(a) + \xi_e \circ_e a + \eta_e \circ_e a, \hbox{ for all } a\in E_2(e).$$ The remaining hypothesis on $T$ assures that $(P_2(e)T(e))^{*_e} = -P_2(e)T(e)$, equivalently $$-\xi_e  - \eta_e = -(D_e(e) + \xi_e  + \eta_e ) =-P_2(e) T(e) = (P_2(e)T(e))^{*_e}$$ $$= (D_e(e) + \xi_e  + \eta_e )^{*_e} = (\xi_e  + \eta_e )^{*_e} = \xi_e  - \eta_e, $$ witnessing that $\xi_e  =0.$ Therefore, $P_2(e) T(a) = D_e(a) + \eta_e \circ_e a,$ for all $a\in E_2(e),$ which shows that $P_2(e) T|_{E_2(e)}$ is a triple derivation (cf. the comments in page \pageref{eq Jordan and triple derivations are related by HMPR} or \cite[Lemmata 1 and 2]{HoMarPeRu}). Proposition \ref{p first characterization of triple derivations}$(c)\Rightarrow (a)$ proves that $T$ is a triple derivation.
\end{proof}

It should be remarked that JBW$^*$-triples and compact JB$^*$-triples satisfy the hypotheses of the above theorem (cf. \cite[Lemma 3.11]{Horn87} and \cite{BuChu92}).

\section{Linear maps on a JB$^*$-triple which are triple derivable at orthogonal elements}\label{sec:5}

This section is devoted to the study of those continuous linear maps on a JB$^*$-triple which are triple derivable at orthogonal elements, that is, Conjecture \ref{conjecture 1} in full generality. Our aim is to establish a new result on preservers by providing sufficient conditions on a continuous linear map which is triple derivable at orthogonal pairs to be a triple derivation. As we shall see in this section, the triple setting will provide several surprises.\smallskip 

We recall some concepts. A subset $\mathcal{S}$ in a JB$^*$-triple $E$ will be called \emph{orthogonal} if $0\notin \mathcal{S}$ and $a\perp b$ for all $a\neq b$ in $\mathcal{S}$. The \emph{rank} of $E$, denoted by $r = r(E)$, will be the minimal cardinal number satisfying card$(\mathcal{S}) \leq r,$ for every orthogonal subset $\mathcal{S}\subset E$. As in the case of the contractive projection problem for real JB$^*$-triples (see \cite{Stacho2001}), and the study of linear maps between JB$^*$-triples preserving orthogonality (cf. \cite{BurFerGarMarPe}), the existence of rank one JB$^*$-triples makes Conjecture \ref{conjecture 1} invalid in this case, because every bounded linear operator on a rank one JB$^*$-triple is trivially  triple derivable at orthogonal elements. Let us see an example. A \emph{Cartan factor of type 1} is a JB$^*$-triple which coincides with the space $B(H,K)$, of all bounded linear operators between two complex Hilbert spaces $H$ and $K$, equipped with the triple product given in \eqref{eq product operators}. In the case $K=\mathbb{C}$, the JB$^*$-triple $B(H,\mathbb{C})$ identifies with $H$ under the triple product \begin{equation}\label{eq triple product rank one Cartan factors as Hilbert} \{a,b,c\} = \frac12 (\langle a | b\rangle c + \langle c | b\rangle a)\ \  (a,b,c\in H),
 \end{equation}where $\langle .| .\rangle$ denotes the inner product of $H$. The type 1 Cartan factor $B(H,\mathbb{C})$ is an example of a rank one JB$^*$-triple.\smallskip

We continue with a more detailed interpretation of Theorem \ref{t bd linear maps triple derivable at orth pairs on a JBstar algebra}.\smallskip

Let $z$ be a symmetric element in the centre of a unital JB$^*$-algebra $A$, and let $e$ be a tripotent in $A$. It is easy to check that
$$\{a, z\circ b, c\} = (a\circ (b\circ z)^*)\circ c + (c\circ (b\circ z)^*)\circ a - (a\circ c)\circ (b\circ z)^* $$
$$=z\circ( (a\circ b^*)\circ c + (c\circ b^*)\circ a - (a\circ c)\circ b^* )= z\circ \{a,b,c\}= \{ z\circ a,b,c\},$$ for all $a,b,c\in A$. Therefore, $\{e,z,e\} = 2 (e\circ z)\circ e - e^2 \circ z  = e^2 \circ z,$
$$P_2(e) (z\circ e) = Q(e) (z \circ e) = \{e,\{e,z\circ e, e\},e\} = z\circ \{e,\{e, e, e\},e\}= z\circ e,$$
and
$$Q(e) (z\circ e)  = \{e, z\circ e,e \} = z \circ \{e,e,e\} = z\circ e,$$ witnessing that $z\circ e$ is a symmetric element in the JB$^*$-algebra $A_2(e)$.

\begin{lemma}\label{l A centre and a tripotent} Let $T: A\to A$ be a bounded linear map on a unital JB$^*$-algebra. Suppose $T$ is triple derivable at orthogonal pairs. Let $e$ be a tripotent in $A$. Then the mapping $T_e = P_2(e) T|_{A_2(e)} : A_2(e) \to A_2(e)$ is triple derivable at orthogonal pairs {\rm(}see Lemma \ref{l 1 for linear maps triple derivable at orthogonal pairs}{\rm)}. Let $\delta: A\to A,$ $\delta_e: A_2(e) \to A_2(e),$ $\xi\in Z(A)_{sa}$ and $\xi_e\in Z(A_2(e))_{sa}$ denote the triple derivations and the elements satisfying $T(a) = \delta(a) + \xi\circ a$ {\rm(}$a\in A${\rm)} and $T_e(a) = \delta_e(a) + \xi_e\circ a$ {\rm(}$a\in A_2(e)${\rm)} whose existence is guaranteed by Theorem \ref{t bd linear maps triple derivable at orth pairs on a JBstar algebra}. Then $\xi\circ e = \xi_e.$
\end{lemma}

\begin{proof} Since, by Theorem \ref{t bd linear maps triple derivable at orth pairs on a JBstar algebra}, $T(a) = \delta(a) + \xi\circ a$ {\rm(}$a\in A${\rm)}, and thus $T_e(a) = P_2(e) T(a) = P_2(e) \delta(a) + P_2(e)(\xi\circ a)$ {\rm(}$a\in E_2(a)${\rm)}. On the other hand, by Theorem \ref{t bd linear maps triple derivable at orth pairs on a JBstar algebra}, $T_e(a) = \delta_e(a) + \xi_e\circ a$ {\rm(}$a\in A_2(e)${\rm)}. Lemma \ref{l 1} proves that $P_2(e) \delta|_{E_2(e)}$ is a triple derivation on $E_2(e)$. Therefore
\begin{equation}\label{eq a 08 09} \delta_e(e) + \xi_e = \delta_e(e) + \xi_e\circ_e e =  T_e(e) =  P_2(e) \delta(e) + P_2(e)(\xi\circ e)
 \end{equation} $$=  P_2(e) \delta(e) + P_2(e)\{\xi, 1,e\} =  P_2(e) \delta(e) + P_2(e)(\xi\circ \{1, 1,e\})  $$ $$=  P_2(e) \delta(e) + P_2(e)(\xi\circ e) =  P_2(e) \delta(e) + \xi\circ e,$$ where in the last equality we applied the comments before this lemma assuring that $\xi\circ e$ is a symmetric element in the JB$^*$-algebra $A_2(e)$.\smallskip

As we have commented before, since $\delta_e$ and $P_2(e) \delta|_{A_2(e)}$ are triple derivations on the unital JB$^*$-algebra $A_2(e)$, $\delta_e(e)^{*_e} = -\delta_e(e)$ and $ (P_2(e) \delta(e))^{*_e} = -  P_2(e) \delta(e)$ (cf. \cite[Lemma 1 and its proof]{HoMarPeRu}). By combining this observation with the identity in \eqref{eq a 08 09} we arrive at $\xi\circ e = \xi_e.$
\end{proof}

Along the rest of this section $T:M\to M$ will stand for a bounded linear operator on a JBW$^*$-triple and we shall assume that $T$ is triple derivable at orthogonal pairs. Lemma \ref{l 1 for linear maps triple derivable at orthogonal pairs} guarantees that for each tripotent $e\in M$ the mapping $T_e := P_2(e) T|_{M_2(e)} : M_2(e) \to M_2(e)$ is triple derivable at orthogonal pairs. Since $M_2(e)$ is a (unital) JBW$^*$-algebra, Theorem \ref{t bd linear maps triple derivable at orth pairs on a JBstar algebra} (see also Remark \ref{r unital case for Theorem for JBstar algebras}) there exist a triple derivation $\delta_e : M_2(e)\to M_2(e)$ and a (unique) central symmetric element $\xi_e\in Z(M_2(e))_{sa}$ such that $$T_e (a) = P_2(e) T(a) =\delta_e(a) + \xi_e\circ_e a, \hbox{ for all $a\in M_2(e)$}.$$ We shall try to define a mapping $S$ on $M$ defined by the elements $\xi_e$. Let us observe that, by Lemma \ref{l 1 for linear maps triple derivable at orthogonal pairs} and Theorem \ref{t bd linear maps triple derivable at orth pairs on a JBstar algebra}, $$\xi_e =\frac{P_2(e)T(e) + (P_2(e)T(e))^{*_e}}{2}= \frac{1}{2}(P_2(e) +Q(e)) T(e),$$ for every tripotent $e\in M$.\smallskip

An element $a$ in $M$ will be called \emph{algebraic} if it can be written as a finite (positive) combination of mutually orthogonal tripotents in $M$, that is, $a = \sum_{j=1}^{m} \lambda_j e_j$ with $\lambda_j>0$ and $e_1,\ldots, e_m$ mutually orthogonal tripotents in $M$. 

\begin{lemma}\label{l well definition on algebraic elements} Let $e_1,\ldots, e_{m_1}$ and $v_1,\ldots, v_{m_2}$ be two families of mutually orthogonal tripotents in $M$ and let $\lambda_j, \mu_k$ be positive real numbers such that $\sum_{j=1}^{m_1} \lambda_j e_j = \sum_{k=1}^{m_2} \mu_k v_k$. Then
$\sum_{j=1}^{m_1} \lambda_j \xi_{e_j} = \sum_{k=1}^{m_2} \mu_k \xi_{v_k}.$
\end{lemma}

\begin{proof} Let $b =\sum_{j=1}^{m_1} \lambda_j e_j = \sum_{k=1}^{m_2} \mu_k v_k$. The range tripotent of $b$ coincides with the element $u = \sum_{j=1}^{m_1} e_j = \sum_{k=1}^{m_2} v_k$. Clearly, $u \geq e_j, v_k$ for all $j,k$. By a new application of Theorem \ref{t bd linear maps triple derivable at orth pairs on a JBstar algebra} we deduce the existence of a triple derivation $\delta_u : M_2(u)\to M_2(u)$ and a (unique) central symmetric element $\xi_u\in Z(M_2(u))_{sa}$ such that $$T_u (a) = P_2(u) T(a) =\delta_u(a) + \xi_u\circ_u a, \hbox{ for all $a\in M_2(u)$}.$$ Evaluating at the element $b\in M_2(u)$ we get
$$ \sum_{j=1}^{m_1} \lambda_j  \ \xi_{e_j} = \sum_{j=1}^{m_1} \lambda_j  \ \xi_u\circ_u e_j = \xi_u\circ_u \left( \sum_{j=1}^{m_1} \lambda_j e_j \right) $$  $$=  \xi_u\circ_u \left( \sum_{k=1}^{m_2} \mu_k v_k \right) =  \sum_{k=1}^{m_2} \mu_k\ \xi_u\circ_u v_k =  \sum_{k=1}^{m_2} \mu_k\ \xi_{v_k},$$ where in the first and last equalities we applied Lemma \ref{l A centre and a tripotent}.
\end{proof}

The last ingredient to define a mapping $S:M\to M$ satisfying $S(e) = \xi_e$ for every tripotent $e\in M$ is the next lemma.

\begin{lemma}\label{l definition of S} Let $e$ and $v$ be two tripotents in $M$, and let $a\in M$ be an element such that $a\in M_2(e)\cap M_2(v)$. Then $\xi_v \circ_v a = \xi_e\circ_e a$.
\end{lemma}

\begin{proof} Theorem \ref{t bd linear maps triple derivable at orth pairs on a JBstar algebra} implies the existence of triple derivations $\delta_e : M_2(e)\to M_2(e)$ and $\delta_v : M_2(v)\to M_2(v)$ and (unique) central symmetric elements $\xi_e\in Z(M_2(e))_{sa}$ and $\xi_v\in Z(M_2(v))_{sa}$ such that $$T_e (b) = P_2(e) T(b) =\delta_e(b) + \xi_e\circ_e b, \hbox{ for all $b\in M_2(e)$},$$ and
$$T_v (c) = P_2(v) T(c) =\delta_v(c) + \xi_v\circ_v b, \hbox{ for all $c\in M_2(v)$}.$$

Let $u= r(a)$ denote the range tripotent of $a$ in $M$. It follows from the hypotheses that $u\in M_2(e)\cap M_2(v),$ and by the Peirce arithmetic $M_2(u)\subseteq M_2(e)\cap M_2(v).$ Since $M_2(u)$ is a JBW$^*$-triple, its set of tripotents is norm-total. Therefore, for each $\varepsilon >0$ we can find a finite family of mutually orthogonal tripotents $\{e_1,\ldots,e_{m_1}\}\subset M_2(u)$ and positive real numbers $\lambda_j$ such that 
$$\left\| \xi_e \circ_e a - \sum_{j=1}^{m_1} \lambda_j \xi_e \circ_e e_j \right\|<\frac{\varepsilon}{2}, \hbox{ and }   \left\| \xi_v \circ_v a - \sum_{j=1}^{m_1} \lambda_j \xi_v \circ_v e_j \right\|<\frac{\varepsilon}{2}.$$

Now, we deduce from Lemma \ref{l A centre and a tripotent}, applied to the JBW$^*$-algebras $M_2(e)$ and $M_2(v)$, that $\xi_e \circ_e e_j = \xi_{e_j}$ and $\xi_v \circ_v e_j = \xi_{e_j}$ for all $j\in \{1,\ldots, m_1\}$. It then follows from the previous inequalities that $\left\| \xi_e \circ_e a - \xi_v \circ_v a \right\| <\varepsilon$. the arbitrariness of $\varepsilon >0$ assures that $\xi_e \circ_e a = \xi_v \circ_v a$, which concludes the proof.
\end{proof}

For each $a\in M$, let us define $S(a) := \xi_e \circ_e a$, where $e$ is any tripotent in $M$ such that $a\in M_2(e)$. Lemma \ref{l definition of S} proves that the assignment $a\mapsto S(a)$ gives a well-defined mapping $S: M\to M$. It follows from the definition that $S(e) = \xi_e$ for each tripotent $e\in M$. In particular, $S|_{M_2(e)}: M_2(e) \to M_2(e)$ is a bounded linear mapping. Since $\xi_e$ is an element in the centre of $M_2(e)$, the mapping $S|_{M_2(e)}$ is in the centroid of $M_2(e)$ (a conclusion which is consistent with the observations we made in page \pageref{centroid and Peirce 2-subspaces}). The linearity of the mapping $S$ on the whole $M$ is not an obvious property, actually we shall only culminate our study of Conjeture \ref{conjecture 1} by assuming the following extra property:\smallskip

Let $N$ be a JBW$^*$-triple. We shall say that that \emph{linearity on $N$ is determined by Peirce 2-subspaces} if every mapping $S:N\to N$ such that $P_2(e) S|_{N_2(e)}: N_2(e) \to N_2(e)$ is a linear operator for every tripotent $e\in N$ must be a linear mapping. Every JBW$^*$-triple $N$ admiting a unitary tripotent $u$ (i.e. $N_2(u) = N$) clearly satisfies this property.\smallskip

Let us justify this property or relate it with the example we gave at the beginning of this section. A non-zero tripotent $e$ in a JB$^*$-triple $E$ is called minimal if $E_2(e) = \mathbb{C} e$.  Let $H$ be a complex Hilbert spaces regarded as a type 1 Cartan factor with the product in \eqref{eq triple product rank one Cartan factors as Hilbert}. It is known, and easy to check, that tripotents in $H$ are precisely the elements in the unit sphere of $H$ (and they are all minimal tripotents). Let $S:H\to H$ be a non-linear 1-homogeneous mapping on $H$, i.e., $S(\lambda x) = \lambda S(x)$ for all $x\in H$, $\lambda\in \mathbb{C}$. Since for each tripotent $e\in S(H)$, $H_2(e) = \mathbb{C} e$, and $P_2(e) (x) = \langle x| e\rangle e$, it is easy to see that $P_2(e) S(\lambda e) = \lambda \langle S(e)|e\rangle e$ and hence $P_2(e) S|_{H_2(e)}$ is linear, a property which is not enjoyed by $S$.\smallskip

\begin{theorem}\label{t triple derivable maps on JBWstar triples} Let $T: M\to M$ be a bounded linear mapping on a JBW$^*$-triple $M$. Suppose that linearity on $M$ is determined by Peirce 2-subspaces. Then the following statements are equivalent:\begin{enumerate}[$(a)$]\item $T$ is triple derivable at orthogonal pairs;
\item There exists a triple derivation $\delta: M\to M$ and an operator $S$ in the centroid of $M$ such that $T = \delta + S$.
\end{enumerate}
\end{theorem}

\begin{proof} The implication $(b)\Rightarrow (a)$ holds for any JB$^*$-triple $M$ without any extra assumption.\smallskip

$(a)\Rightarrow (b)$ Let $S:M\to M$ be the mapping defined by Lemma \ref{l definition of S} and subsequent comments. We observe that, for each tripotent $e\in M$, $S(M_2(e))\subseteq M_2(e)$ by definition. Since, the mapping $S|_{M_2(e)}: M_2(e)\to M_2(e)$ is linear for every tripotent $e\in M$, it follows from the hypothesis on $M$ that $S$ is a linear mapping.\smallskip

We shall next prove that $S$ is continuous. Having in mind that, for each tripotent $e\in M$, we have $\xi_e= S(e) = \frac{1}{2}(P_2(e) + Q(e)) T(e)$ we deduce that $\|\xi_e \| \leq \|T\|$. Given $a$ in $M$ and any tripotent $u\in M$ such that $a\in M_2(u)$ we know that $S(a) = a\circ_u \xi_u,$ and hence $\|S(a)\| \leq \|T\| \ \|a\|$, which proves the continuity of $S$.\smallskip

We shall next show that $S$ lies in the centroid of $M$. For this purpose, let us fix a tripotent $e$ in $M$ and an arbitrary $a\in M$. Since Peirce subspaces $M_2(e)$, $M_1(e)$ and $M_0(e)$ are JBW$^*$-subtriples, the elements $P_2(e)(a)$, $P_1(e)(a)$, and $P_0(e)(a)$ can be approximated in norm by finite positive combinations of mutually orthogonal tripotents in the corresponding Peirce subspace. Suppose $\sum_{k=1}^{m_j} \lambda_k e_k$ is a positive combination of mutually orthogonal tripotents in $M_j(e)$. We shall deal first with the case $j=0,2$. Since $\xi_{e_k}\in M_2 (e_k)$ and $e_k\in M_j(e)$, we deduce from Peirce arithmetic that $\xi_{e_k}\in M_2 (e_k)\subseteq M_j(e)$. It follows from the definition of $S$ that \begin{equation}\label{eq 0921a} S\left(\sum_{k=1}^{m_j} \lambda_k e_k\right) = \sum_{k=1}^{m_j} \lambda_k \xi_{e_k}\in M_j(e), \hbox{ for } j=0,2 \hbox{ and } 1\leq k\leq m_j.
\end{equation} The continuity of $S$ implies that \begin{equation}\label{eq 0922c} S(P_j(e)(a))\in M_j (e), \hbox{ for every tripotent } e\in M, \hbox{ and } j=0,2.
\end{equation}

The case $j=1$ is more delicate. We assume first that $e$ is a complete tripotent in $M$. Let us take a tripotent $v$ in $M_1(e)$. We write $y_k = P_k(e) (\xi_{v})$ for $k =0, 1,2$. Since  $y_0=0$ and $\xi_{v}\in (M_2(v))_{sa}$, we deduce that $$ y_2 + y_1 = \xi_{v} = \{v,v,\xi\} = \{v,\xi_{v},v\}.$$ Now, by Peirce arithmetic with respect to $e$ we get $$y_2 = \{v,v,y_2\} =\{v,y_0,v\} = 0, \hbox{ and } y_1 = \{v,v,y_1\} =\{v,y_1,v\}.$$ It follows that $\xi_{v} = y_1 = P_1(e) (\xi_v)\in M_1(v)$. Since tripotents in the JBW$^*$-triple $M_1(e)$ are norm total, by employing the definition of $S$ and an identity like in \eqref{eq 0921a}, a similar argument to that given above implies that $S(M_1(e)) \subseteq M_1(e)$ and therefore \begin{equation}\label{eq 0922b} S(M_j(e))\in M_j (e), \hbox{ for all complete tripotent } e \in M \hbox{ and } j=0,1,2.
\end{equation}

We shall next show that \eqref{eq 0922b} holds for all tripotent $v\in M$. Let us fix a tripotent $v\in M$. By \cite[Lemma 3.12]{Horn87} there exists a complete tripotent $e\in M$ such that $v\leq e$. We can assume that $v$ is non-complete, and hence $e-v\neq 0$. Let us make some observations. Since $v\in M_2(e)$ the Peirce projections $P_j(e)$ and $P_k(v)$ commute for all $j,k\in \{0,1,2\}$. The Peirce 1-subspace $M_1(v)$ decomposes in the form
\begin{equation}\label{eq decomposition 0922}\begin{aligned} M_1 (v) &= \left(M_2(e) \cap M_1(v)\right) \oplus  \left(M_1(e) \cap M_1(v)\right) \\
&= \left(M_2(e) \cap M_1(v) \cap M_1(e-v)\right) \oplus \left(M_0(e-v) \cap M_1(e)\right),
\end{aligned}
 \end{equation} the first equality being clear because $e$ is complete and the corresponding Peirce projections commute.  For the second one we observe that if $x\in M_2(e) \cap M_1(v)$, by orthogonality we have $$x = \{e,e,x\} = \{v,v,x\} + \{e-v,e-v,x\} = \frac12 x +\{e-v,e-v,x\},$$ which shows that $\frac12 x =\{e-v,e-v,x\},$ and thus $x\in M_1(e-v)$.  Assume next that $x\in M_1(e) \cap M_1(v)$. That is, $$\frac12 x = \{e,e,x\} = \{v,v,x\} + \{e-v,e-v,x\} = \frac12 x + \{e-v,e-v,x\},$$ witnessing that $x\in M_0(e-v)$.  Finally, if $x\in M_0(e-v) \cap M_1(e)$, again by orthogonality we have $$\frac12 x = \{e,e,x\} = \{v,v,x\} + \{e-v,e-v,x\} = \{v,v, x\},$$ which implies that $x\in M_1(v)$. \smallskip
 
The summands $W_1= M_2(e) \cap M_1(v) \cap M_1(e-v)$ and $W_2 = M_0(e-v) \cap M_1(e)$ are JBW$^*$-subtriples of $M$ and $M_1(v) = W_1\oplus W_2$. By \eqref{eq 0922b} $S(M_1(e))\subseteq M_1(e)$ because $e$ is complete; and by \eqref{eq 0921a} $S(M_0(e-v) )\subseteq M_0(e-v)$. Therefore \begin{equation}\label{eq S maps W2 into itself } S(W_2) \subseteq M_0(e-v) \cap M_1(e) = W_2.
\end{equation} 

Take now a tripotent $w$ in the JBW$^*$-triple $W_1= M_2(e) \cap M_1(v) \cap M_1(e-v)$. In this case, by Lemma \ref{l A centre and a tripotent}, applied to the JBW$^*$-algebra $M_2(e)$ and the tripotents $v, e-v$ and $e$, we get $\xi_{v} = \xi_e\circ_e v$, $\xi_{e-v} = \xi_e\circ_e (e-v)$, $$\xi_{e} = \xi_e\circ_e e = \xi_e\circ_e v + \xi_e\circ_e (e-v) = \xi_{v} +\xi_{e-v},$$ with $v, \xi_{v} \perp \xi_{e-v}, e-v,$ because $v$ and $e-v$ are two orthogonal projections in $M_2(e)$ and $\xi_e$ is a symmetric central element in the latter JBW$^*$-algebra. Lemma  \ref{l A centre and a tripotent} also implies that
$$\xi_{w} = \xi_e\circ_e w = \{\xi_e,e,w\} \in M_2(e),$$ and by the above properties 
$$ \{\xi_e,e,w\} = \{\xi_e,v,w\} + \{\xi_e,e-v,w\}= \{\xi_v,v,w\} + \{\xi_{e-v},e-v,w\}.$$ By Peirce arithmetic $$\{\xi_v,v,w\}\in M_{2-2+1}(v)\cap M_{0-0+1} (e-v) = M_1(v) \cap M_1(e-v),$$ and 
$$\{\xi_{e-v},e-v,w\}\in M_{2-2+1}(e-v)\cap M_{0-0+1} (v) = M_1(v) \cap M_1(e-v).$$ Therefore, $$\xi_{w} \in M_2(e) \cap M_1(v) \cap M_1(e-v)= W_1,$$ for every tripotent $w\in W_1$. Having in mind that the set of tripotents in the JBW$^*$-triple $W_1$ is norm-total, a linearization like the one in \eqref{eq 0921a} and the norm continuity and linearity of $S$ assert that $S(W_1)\subseteq W_1$. Combining the latter conclusion with \eqref{eq S maps W2 into itself }, \eqref{eq decomposition 0922} and the linearity of $S$ we get $S(M_1(v)) \subseteq W_1\oplus W_2 = M_1(v)$.  It then follows that \begin{equation}\label{eq 0922 Peirce 1-subspaces for general trip} S(M_j(v))\in M_j (v), \hbox{ for every tripotent } v\in M, \hbox{ and } j=0,1,2,
\end{equation} (cf. \eqref{eq 0922c} for $j=0,2$).\smallskip

Consequently, for each tripotent $v\in M$, $$S L(v,v)(a)  = S(P_2(v) (a) + \frac12 P_1(v) (a)) $$ $$= S(P_2(v) (a)) + \frac12 S(P_1(v) (a)) =L(v,v) S(a).$$ Since the set of tripotents is norm-total in $M$ and $S$ is continuous, the identity $S L(a,a) = L(a,a) S$ holds for every $a\in M$, witnessing that $S$ is an element in the centroid of $M$ (cf. \cite[pages 330, 331]{DiTim88}).\smallskip

Finally, the linear mapping $\delta = T-S: M\to M$ is linear and continuous. Furthermore, $\delta$ is triple derivable at orthogonal pairs because $T$ and $S$ are (cf. Lemma \ref{l properies of the conjugate}$(b)$). It follows from the definition of $S$ (see also Theorem \ref{t bd linear maps triple derivable at orth pairs on a JBstar algebra}) that for each tripotent $e\in M$, we have $$ \frac{1}{2}(P_2(e) + Q(e)) \delta (e) = \frac{1}{2}(P_2(e) + Q(e)) (T-S)(e) = \xi_e -\xi_e=0. $$ Theorem \ref{t characterization of triple derivations} assures that $\delta $ is a triple derivation. Clearly $T = \delta +S$.
\end{proof}

A subspace $I$ of a JB$^*$'triple $E$ is called a \emph{triple ideal} if $\{E,E,I\}+ \{I,E,E\}\subseteq I$. A JBW$^*$-triple which cannot be decomposed as a  direct orthogonal sum of two non-trivial ideals is called a JBW$^*$-triple factor (cf. \cite[$(4.7)$]{Horn87}). The centroid of a JBW$^*$-factor reduces to the complex multiples of the identity mapping (see \cite[Corollary 2.11]{DiTim88}). In particular the centroid of each Cartan factor is one-dimensional.\smallskip

The simplicity of the centroid in the case of Cartan factors and atomic JBW$^*$-triples allows us to relax some of the hypothesis in the previous Theorem \ref{t triple derivable maps on JBWstar triples}. An atomic JBW$^*$-triple is a JBW$^*$-triple which coincides with the $\ell_\infty$-sum of a family of Cartan factors. We have already presented the Cartan factors of type 1 at the beginning of this section, the definition of the remaining Cartan factors reads as follows: Suppose $j$ is a conjugation (i.e. a conjugate linear isometry of period 2) on a complex Hilbert space $H$ and define a linear involution on $B(H)$ given by $x\mapsto x^t:=jx^*j$. The JB$^*$-subtriple of $B(H)$ of all $t$-skew-symmetric (respectively, $t$-symmetric) operators in $B(H)$ is called a Cartan factor of type 2 (respectively, of type 3). A Cartan factor of type 4, is a a complex Hilbert space provided
with a conjugation $x\mapsto \overline{x},$ where triple product and the norm are  given by $$\{x, y, z\} = (x|y)z + (z|y) x -(x|\overline{z})\overline{y},$$ and $ \|x\|^2 = (x|x) + \sqrt{(x|x)^2 -| (x|\overline{x}) |^2},$ respectively. The \emph{Cartan factors of types 5 and 6} (also called \emph{exceptional} Cartan factors) are spaces of matrices over the eight dimensional complex algebra of Cayley numbers; the type 6 consists of all 3 by 3 self-adjoint matrices and has a natural Jordan algebra structure, and the type 5 is the subtriple consisting of all 1 by 2 matrices (see \cite{Ka97} for a detailed presentation of Cartan factors).\smallskip

We recall that two tripotents $u,v$ in a JB$^*$-triple $E$ are called \emph{collinear} (written $u\top v$) if $u\in E_1(v)$ and $v\in E_1(u)$. We say that $u$ \emph{governs} $v$ ($u \vdash v$ in short) whenever $v\in U_{2} (u)$ and $u\in U_{1} (v)$. According to \cite{DanFri87,Neher87}, an ordered quadruple $(u_{1},u_{2},u_{3},u_{4})$ of tripotents in a JB$^*$-triple $E$ is called a \emph{quadrangle} if $u_{1}\bot u_{3}$, $u_{2}\bot u_{4}$, $u_{1}\top u_{2}$ $\top u_{3}\top u_{4}$ $\top u_{1}$ and $u_{4}=2 \{
{u_{1}},{u_{2}},{u_{3}}\}$ (the latter equality also holds if the indices are permutated cyclically, e.g. $u_{2} = 2 \{{u_{3}},{u_{4}},{u_{1}}\}$). An ordered triplet $ (v,u,\tilde v)$ of tripotents in $E$, is called a \emph{trangle} if $v\bot \tilde v$, $u\vdash v$, $u\vdash \tilde v$ and $ v = Q(u)\tilde v$.\smallskip

We say that a tripotent $u\in E$ has rank two if it can be written as the sum of two orthogonal minimal tripotents.

\begin{theorem}\label{t triple derivable maps on atomic JBWstar triples} Let $T: M\to M$ be a bounded linear mapping on an atomic JBW$^*$-triple $M$.
Suppose that $M = \bigoplus_{j}^{\infty} C_j$, where each $C_j$ is a Cartan factor with rank at least $2$.
Then the following statements are equivalent:\begin{enumerate}[$(a)$]\item $T$ is triple derivable at orthogonal pairs;
\item There exists a triple derivation $\delta: M\to M$ and an operator $S$ in the centroid of $M$ such that $T = \delta + S$.
\end{enumerate}
\end{theorem}

\begin{proof} We have already commented that $(b)\Rightarrow (a)$ holds for any JB$^*$-triple $M$ without any extra assumption.\smallskip

$(a)\Rightarrow (b)$ We consider again the mapping  $S:M\to M$ defined by Lemma \ref{l definition of S} and subsequent comments. We observe that $S(M_2(e))\subseteq M_2(e)$ by definition. We have seen above that for each tripotent $e\in M$ the mapping $S|_{M_2(e)}: M_2(e)\to M_2(e)$ is a bounded linear operator in the centroid of $M_2(e)$ (see the comments after Lemma \ref{l definition of S}).\smallskip

By \cite[Proposition 3]{HoMarPeRu} Cartan factors of type 1 with dim$(H)=$dim$(K)$, Cartan factors of type 2 with dim$(H)$ even or infinite and all Cartan factors of type 3 admit a unitary element, and hence they are unital JBW$^*$-algebras. Clearly, the Cartan factor of type 6 also admits a unitary element and enjoys the same structure. Suppose we can find a unitary element $u_j\in C_j$. In this case, $P_2(u_j) T|_{(C_j)_2(u_j)=C_j}: (C_j)_2(u_j)= C_j \to C_j$ is linear, continuous and triple derivable at orthogonal pairs with $\| P_2(u_j) T|_{C_j} \|\leq \|T\|$. It follows from Theorem \ref{t bd linear maps triple derivable at orth pairs on a JBstar algebra} that there exists a triple derivation $\delta_j : C_j \to C_j$ and a symmetric bounded linear mapping $S_j$ in the centroid of $C_j$ such that $P_2(u_j) T|_{C_j} = \delta_j + S_j$. Since $C_j$ is a factor $S_j = \alpha_j Id_{C_j}$ for a real number $\alpha_j$ with $|\alpha_j| \leq \|T\|$, and \begin{equation}\label{eq decomposition when there exists a unitary} P_2(u_j) T (a) = \delta_j (a) + \alpha_j a, \hbox{ for all } a\in C_j.
\end{equation}

We suppose now that $C_j$ does not contain any unitary element. Consider two minimal tripotents $e$ and $v$ in $C_j$. By \cite[Lemma 3.10]{FerPe18Adv} one of the following statements holds:\begin{enumerate}[$(a)$]\item There exist minimal tripotents $v_2,v_3,v_4$ in $C_j$, and complex numbers $\alpha$, $\beta$, $\gamma$, $\delta$ such that $(e,v_2,v_3,v_4)$ is a quadrangle, $|\alpha|^2 +| \beta|^2 + |\gamma|^2 + |\delta|^2 =1$, $\alpha \delta  = \beta \gamma$, and $v = \alpha e + \beta v_2 + \gamma v_4 + \delta v_3$;
\item There exist a minimal tripotent $\tilde e\in C_j$, a rank two tripotent $u\in C_j$, and complex numbers $\alpha, \beta, \delta$ such that $(e, u,\tilde e)$ is a trangle, $|\alpha|^2 +2 | \beta|^2 + |\delta|^2 =1$, $\alpha \delta  = \beta^2$, and $v = \alpha e+ \beta u +\delta \tilde e$.
\end{enumerate} This result, in particular, implies that for each rank two tripotent $u$ in a Cartan factor $C$ of rank at least two the Peirce subspace $C_2(u)$ is again a factor. It follows from the above that we can find a rank two tripotent $u\in C_j$ such that $e,v\in (C_j)_2(u)$. By Lemma \ref{l 1 for linear maps triple derivable at orthogonal pairs} the mapping $P_2(u) T|_{M_2(u)= (C_j)_2(u)} : (C_j)_2(u)\to (C_j)_2(u)$ is a bounded linear mapping which is triple derivable at orthogonal pairs. Theorem \ref{t bd linear maps triple derivable at orth pairs on a JBstar algebra} proves the existence of a triple derivation $\delta_{u}$ on $(C_j)_2(u)$ and a symmetric element $S_u$ in the centroid of $(C_j)_2(u)$ such that $P_2(u) T|_{M_2(u)}= \delta_u +S_u.$ Since $(C_j)_2(u)$ is factor, there exists a real number $\alpha_u$ such that $S_u (x) = \alpha_u x$ for all $x\in (C_j)_2(u)$, that is, $\xi_u = \alpha_u u$. Consequently, by Lemma \ref{l A centre and a tripotent}, $\xi_e = \xi_u \circ_u e= S_u(e)= \alpha_{u} e$ and  $\xi_v= \xi_u \circ_u v= S_u(v) = \alpha_{u} v$. We have therefore concluded that for any couple of minimal tripotents $e,v \in C_j$ $$  \exists  \alpha_{e,v}\in \mathbb{R} \hbox{ such that  } \xi_e = S_u(e)= \alpha_{e,v} e,\hbox{ and } \xi_v= S_u(v) = \alpha_{e,v} v,$$ and thus \begin{equation}\label{eq scalar multiple on minimal tripotents}  \exists  \alpha_{j}\in \mathbb{R} \hbox{ such that  } \xi_e = \alpha_{j} e,\hbox{ for all minimal tripotent } e\in C_j.
\end{equation}

Let $a$ be any element in $C_j$ and let $u$ be a tripotent such that $a\in (C_j)_2(u)$. Actually, $(C_j)_2(u)$ is another Cartan factor (cf. \cite[Lemma 3.9, Structure Theorem 3.14 and Classification Theorem 3.20]{Neher87}). By Lemma \ref{l 1 for linear maps triple derivable at orthogonal pairs} and Theorem \ref{t bd linear maps triple derivable at orth pairs on a JBstar algebra}, there exists a triple derivation $\delta_u$ on $(C_j)_2(u)$ and a symmetric element in the centroid of $(C_j)_2(u)$, that is, a real number $\alpha_u$ such that $$P_2(u) T(a) = \delta_u (a) + \alpha_u a, \hbox{ for all } a\in (C_j)_2(u).$$ When the previous identity is evaluated at a minimal tripotent $e\in (C_j)_2(u)$ we deduce that $\alpha_u = \alpha_j$, equivalently, $\xi_u = \alpha_j u$ (cf. \eqref{eq scalar multiple on minimal tripotents}), and thus $S(a) = \alpha_j a$ for all $a\in C_j$. We observe that $|\alpha_j|\leq \|T\|$ for all $j$.\smallskip

Let $P_j$ denote the projection of $M$ onto $C_j$. Since $P_j$ is a contractive projection, Lemma \ref{l 1 for linear maps triple derivable at orthogonal pairs} assures that $P_j T|_{C_j} : C_j \to C_j$ is triple derivable at orthogonal pairs. The operator $\delta_j = P_j T|_{C_j} -S|_{C_j}$ also is triple derivable at orthogonal pairs, and by the arguments above, $\frac12 (P_2(u)+Q(u)) \delta_j (u) = 0$ for all tripotent $u\in C_j$. Theorem \ref{t characterization of triple derivations} implies that $\delta_j$ is a triple derivation on $C_j$ and \begin{equation}\label{eq decomposition on a Cartan factor without unitaries} P_j T(a) = \delta_j (a) + \alpha_j a, \hbox{ for all } a\in C_j.
\end{equation}

For each two different indices $j_1\neq j_2$, let us take two tripotents $e_{j_1}\in C_{j_1}$ and $e_{j_2}\in C_{j_2}$. Since $e_{j_1}\perp e_{j_2}$ it follows from the hypotheses that $$0 = \{T(e_{j_1}), e_{j_2}, e_{j_2}\} + \{e_{j_1}, T(e_{j_2}), e_{j_2}\}.$$ Let us consider the second summand. Since $e_{j_1}\in M_0(e_{j_2})$, $e_{j_2}\in M_2(e_{j_2})$, by Peirce arithmetic, $\{e_{j_1}, T(e_{j_2}), e_{j_2}\} = \{e_{j_1}, P_1(e_{j_2}) T(e_{j_2}), e_{j_2}\} =0$, because $P_1(e_{j_2}) (M) \subseteq C_{j_2}\subset M_0(e_{j_1})$. This implies that $0 = \{T(e_{j_1}), e_{j_2}, e_{j_2}\},$ which proves that $T(e_{j_1})\perp e_{j_2}$ for every tripotent $e_{j_2}\in C_{j_2}$. The arbitrariness of $j_2\neq j_1$ and the fact that tripotents are norm-total in each $C_j$ assure that $T(C_{j_1}) \subseteq C_{j_1}$ for every $j_1$.\smallskip

Finally, since for each $j$, $T(C_{j}) \subseteq C_{j}$ and $T|_{C_j} : C_j \to C_j$ can be written in the form $T|_{C_j} = \delta_j +\alpha_j Id_{C_j}$ for a triple derivation $\delta_j$ on $C_j$ and a real number $\alpha_j$ with $|\alpha_j|\leq \|T\|$ (cf. \eqref{eq decomposition when there exists a unitary} and \eqref{eq decomposition on a Cartan factor without unitaries}), we deduce that $$T((a_j)_j) = \delta ((a_j)_j)  + S((a_j)_j), $$ where $\delta : M\to M$ is the triple derivation given by $\delta((a_j)_j) = (\delta_j(a_j))_j$ and $S$ is the bounded linear mapping given by $S((a_j)_j)  = (\alpha_j a_j)_j$, which is an element in the centroid of $M$.
\end{proof}

\medskip\medskip

\textbf{Acknowledgements} Second author partially supported by the Spanish Ministry of Science, Innovation and Universities (MICINN) and European Regional Development Fund project no. PGC2018-093332-B-I00, Programa Operativo FEDER 2014-2020 and Consejer{\'i}a de Econom{\'i}a y Conocimiento de la Junta de Andaluc{\'i}a grant number A-FQM-242-UGR18, and Junta de Andaluc\'{\i}a grant FQM375.
\smallskip\smallskip

\end{document}